\documentclass{amsart}

\usepackage{amsmath,amssymb}
\usepackage{mathrsfs}
\usepackage{mathtools}
\usepackage{stmaryrd}
\usepackage{enumerate}
\usepackage{tikz-cd}
\usepackage[all]{xy}
\usepackage{aliascnt}
\usepackage[colorlinks=true, linkcolor=blue, citecolor=blue]{hyperref}

\usepackage{combelow}

\newtheorem{theorem}{Theorem}

\newaliascnt{lemma}{theorem}
\newtheorem{lemma}[lemma]{Lemma}
\aliascntresetthe{lemma}

\newaliascnt{corollary}{theorem}
\newtheorem{corollary}[corollary]{Corollary}
\aliascntresetthe{corollary}

\newaliascnt{proposition}{theorem}
\newtheorem{proposition}[proposition]{Proposition}
\aliascntresetthe{proposition}

\newaliascnt{conjecture}{theorem}

\aliascntresetthe{conjecture}

\newaliascnt{question}{theorem}

\aliascntresetthe{question}

\theoremstyle{definition}

\newaliascnt{definition}{theorem}
\newtheorem{definition}[definition]{Definition}
\aliascntresetthe{definition}

\newaliascnt{remark}{theorem}
\newtheorem{remark}[remark]{Remark}
\aliascntresetthe{remark}

\newaliascnt{example}{theorem}

\aliascntresetthe{example}

\newaliascnt{notation}{theorem}
\newtheorem{notation}[notation]{Notation}
\aliascntresetthe{notation}

\newif\ifhascomments \hascommentstrue
\ifhascomments
  \newcommand{\matt}[1]{{\color{red}[[\ensuremath{\spadesuit\spadesuit\spadesuit} #1]]}}
  \newcommand{\jeremy}[1]{{\color{red}[[\ensuremath{\clubsuit\clubsuit\clubsuit} #1]]}}
\else
  \newcommand{\matt}[1]{}
  \newcommand{\jeremy}[1]{}
\fi

\renewcommand{\setminus}{\smallsetminus}

\newcommand{\Z}{\mathbb{Z}}
\newcommand{\Q}{\mathbb{Q}}
\newcommand{\R}{\mathbb{R}}

\newcommand{\cX}{\mathcal{X}}
\newcommand{\cY}{\mathcal{Y}}
\newcommand{\cZ}{\mathcal{Z}}
\newcommand{\cW}{\mathcal{W}}

\newcommand{\cC}{\mathcal{C}}
\newcommand{\cD}{\mathcal{D}}
\newcommand{\cU}{\mathcal{U}}

\newcommand{\cV}{\mathcal{V}}

\newcommand{\sL}{\mathscr{L}}
\newcommand{\sM}{\mathscr{M}}

\newcommand{\sJ}{\mathscr{J}}
\newcommand{\sW}{\mathscr{W}}

\newcommand{\bL}{\mathbb{L}}

\newcommand{\bG}{\mathbb{G}}
\newcommand{\bA}{\mathbb{A}}

\newcommand{\diff}{\mathrm{d}}
\newcommand{\red}{\mathrm{red}}

\DeclareMathOperator{\het}{ht}

\DeclareMathOperator{\coh}{coh}
\DeclareMathOperator{\e}{e}

\DeclareMathOperator{\Spec}{Spec}

\DeclareMathOperator{\ord}{ord}

\DeclareMathOperator{\EP}{EP}

\DeclareMathOperator{\val}{val}

\DeclareMathOperator{\can}{can}

\newcommand{\cE}{\mathcal{E}}

\DeclareMathOperator{\ordjac}{ordjac}

\newcommand{\tor}{\mathrm{tor}}

\DeclareMathOperator{\coker}{coker}

\tikzset{cong/.style={draw=none,edge node={node [sloped, allow upside down, auto=false]{$\cong$}}},
         Isom/.style={above,every to/.append style={edge node={node [sloped, allow upside down, auto=false]{$\sim$}}}}}

\title{Motivic integration for singular Artin stacks}

\author{Matthew Satriano and Jeremy Usatine}

\thanks{MS was partially supported by a Discovery Grant from the National Science and Engineering Research Council of Canada as well as a Mathematics Faculty Research Chair from the University of Waterloo}

\address{Matthew Satriano, Department of Pure Mathematics, University of Waterloo}
\email{msatriano@uwaterloo.ca}

\address{Jeremy Usatine, Department of Mathematics, Florida State University}
\email{jusatine@fsu.edu}

\begin{document}

\begin{abstract}
Let $\cX \to Y$ be a birational modification of a variety by an Artin stack. In previous work, under the assumption that $\cX$ is smooth, we proved a change of variables formula relating motivic integrals over arcs of $Y$ to motivic integrals over arcs of $\cX$. In this paper, we extend that result to the case where $\cX$ is singular. We may therefore apply this generalized formula to the so-called warping stack $\sW(\cX)$ of $\cX$, which may be singular even when $\cX$ is smooth. We thus obtain a change of variables formula \emph{canonically} expressing any given motivic integral over arcs of $Y$ as a motivic integral over \emph{warped arcs} of $\cX$.
\end{abstract}

\maketitle

\numberwithin{theorem}{section}
\numberwithin{lemma}{section}
\numberwithin{corollary}{section}
\numberwithin{proposition}{section}
\numberwithin{conjecture}{section}
\numberwithin{question}{section}
\numberwithin{remark}{section}
\numberwithin{definition}{section}
\numberwithin{example}{section}
\numberwithin{notation}{section}

\setcounter{tocdepth}{1}

\tableofcontents

\section{Introduction}

Motivic integration was pioneered by Kontesevich \cite{Kontsevich} in his proof that birational smooth Calabi--Yau varieties have equal Hodge numbers. Since then, the field has seen broad ranging applications in birational geometry, mirror symmetry, and the study of singularities. Central to these applications is a \emph{motivic change of variables formula} relating the motivic measures of $Y$ and $X$ under a birational modification $X \to Y$, see e.g., \cite{DenefLoeser1999, Looijenga}. 

Motivated by mirror symmetry for singular Calabi-Yau varieties, Batyrev \cite{Batyrev1998} introduced the \emph{stringy Hodge numbers} for varieties with \emph{log-terminal singularities}. These are defined in terms of the combinatorial information of a resolution of singularities rather than the dimensions of cohomology groups. However when $Y$ admits a \emph{crepant} resolution $X\to Y$, the stringy Hodge numbers of $Y$ agree with their classical counterparts on $X$.

% In \cite{DenefLoeser2002}, Denef and Loeser refined the stringy Hodge numbers by introducing the \emph{Gorenstein measure} $\mu^\Gor_Y$ on the arc scheme $\sL(Y)$; this measure assumes values in a modified Grothendieck ring of varieties and recovers the stringy Hodge numbers after a suitable specialization. By proving a motivic change of variables formula for $\mu^\Gor_Y$, Denef and Loeser obtained a motivic version of the McKay correspondence, refining the original version proved by Batyrev \cite{Batyrev99}.

In \cite{Yasuda2004} (see also \cite{Yasuda2006, Yasuda2019}), Yasuda gave a beautiful description for the stringy Hodge numbers of $Y$ in the case where $Y$ has only quotient singularities. Such varieties need not admit a crepant resolution by a scheme, however Vistoli \cite{Vistoli89} proved that they always admit small (hence crepant) resolutions $\pi\colon\cX\to Y$ by Deligne--Mumford stacks. Yasuda extended the theory of motivic integration to the case of smooth Deligne--Mumford stacks and proved a motivic change of variables for $\pi$. As a result, he obtained that the stringy Hodge numbers of $Y$ coincide with the orbifold Hodge numbers of $\cX$ in the sense of \cite{ChenRuan}.

Since varieties with log-terminal singularities rarely admit crepant resolutions by Deligne--Mumford stacks, if we wish to say something similar for log-terminal singularities in general, we must consider Artin stacks. Indeed, it is easy to write down examples (see e.g., \cite[Example 7.1]{SatrianoUsatine3}) of log-terminal varieties admitting a small (hence crepant) resolution by an Artin stack, yet no crepant resolution by a Deligne--Mumford stack and even no non-commutative crepant resolution (NCCR), as defined by van den Bergh \cite{vdB2004}. Furthermore, we recently proved \cite{SatrianoUsatine3} that \emph{every} log-terminal variety admits a crepant resolution by a smooth Artin stack.

In order to apply a motivic change of variables formula to a resolution $\pi\colon\cX\to Y$ by a stack, one must understand how arcs of $Y$ lift to arcs of $\cX$. Here we must draw an important distinction between three cases of increasing generality:~when $\cX$ is a scheme, when $\cX$ is Deligne--Mumford, or when $\cX$ is Artin. In the scheme case, (outside a set of measure 0) every arc $\varphi\colon D\to Y$ lifts uniquely to $\cX$. When $\cX$ is Deligne--Mumford, $\varphi$ need not lift, however it does lift uniquely to a so-called \emph{twisted arc} as introduced by \cite{Yasuda2004}. When $\cX$ is Artin, the situation is far more complicated:~since Artin stacks are non-separated, $\varphi$ may admit many (even infinitely many) lifts to $\cX$. Thus, one is in need of a more general notion than twisted arcs, one that yields a \emph{unique} lift of $\varphi$.

 %Despite the fact that $\pi$ is not separated, several influential papers \cite{Alper,AHRLuna,AHRetalelocal,AHLH} have shown 
%Nonetheless, work of Alper \cite{Alper}, Alper--Hall--Rydh \cite{AHRLuna,AHRetalelocal}, and Alper--Halpern-Leistner--Heinloth \cite{AHLH} has shown that good moduli space maps behave in many respects like proper coarse space maps:~although they are not separated, they are universally closed, satisfy formal GAGA and a cohomology and base change theorem, as well as have \'etale local structure similar that of coarse spaces. This suggests that good moduli space maps $\cX\to Y$ should satisfy a valuative criterion analogous to that of proper maps, namely one would like arcs of $Y$ to admit unique twisted lifts to $\cX$.
%Unfortunately, twisted curves and arcs are insufficient for the purposes of studying Artin stacks. Indeed, although good moduli spaces $\cX\to Y$ share many features with proper maps, they are not separated. Hence, arcs or curves on $Y$ may admit many (or even infinitely many) twisted lifts to $\cX$. 

In \cite{SatrianoUsatine4}, we construct this desired theory, obtaining new notions of arcs which we call \emph{warped arcs}. We prove that (outside a set of measure 0) every arc $\varphi$ of $Y$ admits a \emph{canonical} warped lift $\varphi^{\can}$, and that such lifts can naturally be interpreted as usual arcs of an auxiliary stack $\sW(\cX)$, see \cite[Theorem 1.8 and Theorem 1.15]{SatrianoUsatine4}. It is thus natural to ask for a motivic change of variables formula that can be applied to $\sW(\cX)$.

A technical hurdle arises here:~even when $\cX$ is smooth and finite type, $\sW(\cX)$ need not be. Thus, even if one is only interested in resolutions by smooth Artin stacks, in order to obtain canonical lifts, one must work with $\sW(\cX)$, and is therefore is need of a more general motivic change of variables formulas.

\vspace{1em}

The purpose of this paper is to accomplish this necessary generalization. Namely, we prove a motivic change of variables formula for stacks that are: (i) singular, (ii) locally of finite type, and (iii) non-equidimensional, see \autoref{theoremMotivicChangeOfVariablesMeasurable}. This generalizes our previous work \cite{SatrianoUsatine2} where we proved a motivic change of variables formula for smooth stacks. Ultimately, we prove our main theorem (\autoref{theoremMotivicChangeOfVariablesMeasurable}) by reducing to the smooth case; thus, the current work relies on \cite{SatrianoUsatine2} rather than supersede it.

To state our main theorem, we must first discuss new kinds of functions that we introduced in \cite{SatrianoUsatine2} known as \emph{height functions}. Let $\sL(\cX)$ denote the stack of arcs of $\cX$ and let $|\sL(\cX)|$ denote its associated topological space. 

Given an object of the derived category $E\in D^-_{\coh}(\cX)$, we let
\[
\het^{(i)}_E\colon|\sL(\cX)|\to\Z\cup\{\infty\}
\]
be the function assigning to an arc $\varphi\colon \Spec k'\llbracket t \rrbracket\to\cX$ the value
\[
\het^{(i)}_E(\varphi)=\dim_{k'} H^i(L\varphi^*E).
\]
This definition was inspired by recent work \cite{ESZB} of the first author, Ellenberg, and Zureick-Brown, where we generalized the notion of number-theoretic Weil heights to the case of stacks, and unified the Batyrev--Manin and Malle Conjectures.

When $\cX$ is smooth, our motivic change of variables formula $\pi\colon\cX\to Y$ is governed by the difference of heights
\[
\het^{(0)}_{L_{\cX/Y}}-\het^{(1)}_{L_{\cX/Y}},
\]
where $L_{\cX/Y}$ is the relative cotangent complex; see \cite[Theorem 1.3]{SatrianoUsatine2} and \cite[Theorem 2.3]{SatrianoUsatine3}. When $\cX$ is a scheme, $\het^{(1)}_{L_{\cX/Y}}$ vanishes and $\het^{(0)}_{L_{\cX/Y}}$ agrees with the order function of the relative Jacobian ideal of $\pi$. Hence, we recover Kontsevich's original motivic change of variables formula.

In this paper, we show that for singular $\cX$, the above height function needs to be modified by an interesting correction term. Namely, we introduce the following \emph{relative height function}:
\begin{align*}
	\het_{\cX/\cY}(\varphi) :=\ &\het^{(0)}_{L_{\cX/\cY}}(\varphi) - \het^{(1)}_{L_{\cX/\cY}}(\varphi)\\
	&- \dim_{k'}\coker\left(H^0(L\varphi^*L\pi^*L_\cY)_\tor \to H^0(L\varphi^* L_\cX)_\tor\right),
\end{align*}
where the subscript $\tor$ denotes the torsion submodule; see Definition \ref{def:relhet}. %We prove in Proposition \ref{prop:additiveht-exact-triangle} that the relative height function is additive under composition of morphisms.

We may now state our main theorem after fixing the following notation. See \autoref{def:muxd} for the definition of the measure $\mu_{\cX,d}$.

\begin{notation}\label{not:mainnot}
Let $\cX$ be a locally finite type Artin stack over $k$ with affine geometric stabilizers and separated diagonal, let $Y$ be an irreducible finite type scheme over $k$, let $\pi: \cX \to Y$ be a morphism, let $U$ be a non-empty smooth open subscheme of $Y$ such that $\cU:=\pi^{-1}(U) \to U$ is an isomorphism. 
\end{notation}

\begin{theorem}\label{theoremMotivicChangeOfVariablesMeasurable}
Keep Notation \ref{not:mainnot}, and let $\cC \subset |\sL(\cX)|$ be a cylinder such that the map 
\[
%\overline{
\cC \setminus |\sL(\cX \setminus \cU)|\ \longrightarrow\ \sL(Y) \setminus \sL(Y \setminus U)
\]
induced by $\pi$ is a bijection on isomorphism classes of $k'$-points for all field extensions $k'$ of $k$. Then 
\[
	 \int_{\sL(Y)} \bL^{f} \diff\mu_Y = \int_{\cC \setminus |\sL(\cX \setminus \cU)|} \bL^{f \circ \sL(\pi)-\het_{\cX/Y}} \diff\mu_{\cX, \dim Y}.
\]
for every measurable function $f\colon \sL(Y) \to \Z \cup \{\infty\}$ where $\bL^f$ is integrable on $\sL(Y)$.
\end{theorem}

\begin{remark}

See Theorem \ref{theoremMotivicChangeOfVariablesMeasurable-realversion} for a stronger version of the above theorem; in particular, Theorem \ref{theoremMotivicChangeOfVariablesMeasurable-realversion}(\ref{theoremMotivicChangeOfVariablesMeasurable-realversion::a}) shows that the left hand side of the above equation converges if and only if the right hand side does. See also Theorem \ref{theoremMCVFCylinder} for a variant.% of Theorem \ref{theoremMotivicChangeOfVariablesMeasurable-realversion}.
\end{remark}

Since $\cX$ may be locally of finite type and non-equidimensional, proving \autoref{theoremMotivicChangeOfVariablesMeasurable} requires some extra bookkeeping. For this we introduce the notions of \emph{small cylinders} and \emph{$d$-convergent cylinders}, see Definitions \ref{def:smallcylinder} and \ref{def:d-measurability-and-measure}. The latter concept essentially involves fixing a ``virtual dimension'' $d$ for $\cX$ and only considering those sets that behave like cylinders in the $d$-dimensional case. This leads naturally to the measure $\mu_{\cX,d}$, see \autoref{def:muxd}.

Applying \autoref{theoremMotivicChangeOfVariablesMeasurable} to $\sW(\cX)$ gives our desired motivic change of variables formula for warped arcs. In the statement below, $\cC_\cX$ denotes the set of canonical warped lifts of $\sL(Y)$ and $\underline{\pi}\colon\sW(\cX)\to Y$ denotes a canonical map induced by $\pi$.

\begin{theorem}[{\cite[Theorem 1.18]{SatrianoUsatine4}}]
\label{thm:canonicalmcvf}
%Let $\cX$ be a locally finite type Artin stack over $k$ with affine diagonal, let $Y$ be an irreducible finite type scheme over $k$, let $\pi: \cX \to Y$ be a morphism, and 
Keep Notation \ref{not:mainnot}, assume $\cX$ has affine diagonal, and assume $\cX \to Y$ is a good moduli space map. For any measurable function $f\colon\sL(Y) \to \Z \cup \{\infty\}$ with $\bL^f$ integrable on $\sL(Y)$, we have
\[
	 \int_{\sL(Y)} \bL^{f} \diff\mu_Y \ =\  \int_{\cC_\cX} \bL^{f \circ \sL(\underline{\pi}) - \het_{\sW(\cX)/Y}} \diff\mu_{\sW(\cX), \dim Y}.
\]
\end{theorem}

Because $\sW(\cX)$ may not be smooth, equidimensional, or finite type, \autoref{theoremMotivicChangeOfVariablesMeasurable} is a crucial ingredient in the proof of \autoref{thm:canonicalmcvf}. Given any resolution $\pi\colon\cX\to Y$, this shows that any motivic integral over $Y$ can be expressed canonically as a motivic integral over warped arcs of $\cX$.

%\renewcommand{\thesection}{A}
%\refstepcounter{section}

\subsection*{Conventions}

Throughout this paper, let $k$ be an algebraically closed field of characteristic 0. For any stack $\cX$ over $k$, we let $|\cX|$ denote its associated topological space, and for any subset $\cC \subset |\cX|$ and field extension $k'$ of $k$, we let $\cC(k')$ (resp. $\overline{\cC}(k')$) denote the category of (resp. set of isomorphism classes of) $k'$-valued points of $\cX$ whose class in $|\cX|$ is contained in $\cC$.

\section{Preliminaries}

In this section, we collect several properties of cylinders which will be useful throughout this paper.

\begin{definition}
Let $\cX$ be a locally finite type Artin stack over $k$, and let $\cC \subset |\sL(\cX)|$. We call $\cC$ a \emph{cylinder} if $\cC = \theta_n^{-1}(\cC_n)$ for some $n \in \Z_{\geq 0}$ and $\cC_n$ a locally constructible subset of $|\sL_n(\cX)|$.
\end{definition}

\begin{remark}
If $\cW$ is a locally finite type Artin stack over $k$ and $\cE \subset |\cW|$, then $\cE$ is locally constructible in $|\cW|$ if and only if for every quasi-compact open substack $\cV$ of $\cW$ we have $\cE \cap |\cV|$ is a constructible subset of $|\cV|$.
\end{remark}

Since our motivic change of variables formula applies to locally finite type stacks which are not necessarily quasi-compact, we require a boundedness assumption on our cylinders. This is given in the following definition.

\begin{definition}\label{def:smallcylinder}
Let $\cX$ be a locally finite type Artin stack over $k$, and let $\cC \subset |\sL(\cX)|$ be a cylinder. We call $\cC$ \emph{small} if $\theta_0(\cC)$ is contained in a quasi-compact subset of $|\cX|$.
\end{definition}

The following shows that boundedness of the $0$-trunction of a cylinder forces boundedness of all truncations.

\begin{proposition}
\label{propositionSmallCylinderHasQuasiCompactConstructibleImages}
Let $\cX$ be a locally finite type Artin stack over $k$, and let $\cC \subset |\sL(\cX)|$ be a small cylinder. If $n \in \Z_{\geq 0}$, then $\theta_n(\cC)$ is a quasi-compact locally constructible subset of $|\sL_n(\cX)|$.
\end{proposition}

\begin{proof}
Because $\theta_0(\cC)$ is contained in a quasi-compact subset of $|\cX|$, by replacing $\cX$ with a quasi-compact open substack that contains $\theta_0(\cC)$, we may assume that $\cX$ is finite type over $k$. Chevalley's Theorem for Artin stacks \cite[Theorem 5.1]{HallRydh17} and \cite[Lemma 3.24]{SatrianoUsatine1} then reduces the proposition to the case where $\cX$ is a scheme, which is well-known, see e.g., \cite[Chapter 5 Corollary 1.5.7(b)]{ChambertLoirNicaiseSebag}.
\end{proof}

\begin{remark}
Let $\cW$ be a locally finite type Artin stack over $k$ with affine geometric stabilizers and $\cE$ be a quasi-compact locally constructible subset of $|\cW|$. Then there exists a quasi-compact open substack $\cV$ of $\cW$ such that $\cE$ is a (constructible) subset of $|\cV|$, allowing us to associate to $\cE$ a class $ \e(\cE) \in \widehat{\sM}_k$. It is straightforward to verify that $\e(\cE)$ does not depend on the choice of $\cV$.

In particular if $\cX$ is a locally finite type Artin stack over $k$ with affine geometric stabilizers and $\cC \subset |\sL(\cX)|$ is a small cylinder, for every $n \in \Z_{\geq 0}$, we have a well defined class $\e(\theta_n(\cC)) \in \widehat{\sM}_k$ by \autoref{propositionSmallCylinderHasQuasiCompactConstructibleImages}.
\end{remark}

\begin{proposition}
\label{propositionSmallCylinderFiniteSubcover}
Let $\cX$ be a locally finite type Artin stack over $k$, and let $\cC \subset |\sL(\cX)|$ be a small cylinder. Then every cover of $\cC$ by cylinders has a finite subcover.
\end{proposition}

\begin{proof}
Let $\{\cC^{(i)}\}_i$ be a collection of cylinders in $|\sL(\cX)|$ such that $\cC \subset \bigcup_i \cC^{(i)}$. Because $\cC$ is small, there exists a quasi-compact open substack $\cX_1$ of $\cX$ such that $\theta_0(\cC) \subset |\cX_1|$. Replacing $\cX$ with $\cX_1$ and each $\cC^{(i)}$ with $\cC^{(i)} \cap |\sL(\cX_1)|$, we assume that $\cX$ is finite type. The result then follows from \cite[Proposition 2.2]{SatrianoUsatine2}.
\end{proof}

We next relate cylinders in a stack to that of a smooth cover.

\begin{lemma}
\label{lemmaCylinderCanBeCheckedInSmoothCover}
Let $\eta: \cW \to \cY$ be a smooth surjective morphism of finite type Artin stacks over $k$, and let $\cD \subset |\sL(\cY)|$. If $\sL(\eta)^{-1}(\cD)$ is a cylinder in $|\sL(\cW)|$, then $\cD$ is a cylinder in $|\sL(\cY)|$.
\end{lemma}

\begin{proof}
By replacing $\cW$ with a smooth cover by a finite type $k$-scheme, we may assume that $\cW = W$ is a scheme.

Set $C = \sL(\eta)^{-1}(\cD)$. We will use throughout this proof that because $\eta$ is smooth and surjective, $\sL(\eta): \sL(W) \to |\sL(\cY)|$ is surjective. In particular, $\cD = \sL(\eta)(C)$. Because $C$ is a cylinder, there exists some $n$ and a constructible subset $C_n \subset \sL_n(W)$ such that $C = \theta_n^{-1}(C_n)$. Because $C \subset \theta_n^{-1}(\theta_n (C)) \subset \theta_n^{-1}(C_n) = C$, we may replace $C_n$ with the constructible set $\theta_n(C)$ and thus assume that $C_n = \theta_n(C)$. By Chevalley's theorem for Artin stacks, it is sufficient to show that
\[
	\cD = \theta_n^{-1}(\sL_n(\eta)(C_n)).
\]
For the first direction of this equality, if $\alpha \in \cD$ then
\[
	\theta_n(\alpha) \in \theta_n(\cD) = \theta_n(\sL(\eta)(C)) = \sL_n(\eta)(\theta_n(C)) = \sL_n(\eta)(C_n).
\]
Thus $\cD \subset \theta_n^{-1}(\sL_n(\eta)(C_n))$. For the other direction, let $\alpha \in \theta_n^{-1}(\sL_n(\eta)(C_n))$ and let $\varphi \in \sL(W)$ be such that $\alpha = \sL(\eta)(\varphi)$. Then
\[
	\sL_n(\eta)(\theta_n(\varphi)) = \theta_n(\sL_n(\eta)(\varphi)) = \theta_n(\alpha) \in \sL_n(\eta)(C_n) = \theta_n(\cD),
\]
so
\[
	\theta_n(\varphi) \in \sL_n(\eta)^{-1}(\theta_n(\cD)) = \theta_n(C) = C_n
\]
where the first equality of sets follows from \cite[Lemma 3.24]{SatrianoUsatine1}. Therefore $\varphi \in \theta_n^{-1}(C_n) = C$, so $\alpha = \sL(\eta)(\varphi) \in \sL(\eta)(C) = \cD$ and we are done.
\end{proof}

The following shows a type of base change result for cylinders.

\begin{lemma}
\label{lemmaPreimageImageImagePreimage}
Let $\cX, \cY, \cW$ be locally finite type Artin stacks over $k$, let $\pi: \cX \to \cY$ and $\eta: \cW \to \cY$ be morphisms, and let $\cC \subset |\sL(\cX)|$. Set $\cZ = \cX \times_\cY \cW$ and let $\rho: \cZ \to \cW$ and $\xi: \cZ \to \cX$ be the projection maps. Then
\[
	\sL(\rho)(\sL(\xi)^{-1}(\cC)) = \sL(\eta)^{-1}(\sL(\pi)(\cC)).
\]
\end{lemma}

\begin{proof}
For the first direction, let $\alpha \in \sL(\rho)(\sL(\xi)^{-1}(\cC))$, and let $\varphi \in \sL(\xi)^{-1}(\cC)$ such that $\alpha = \sL(\rho)(\varphi)$. Then
\[
	\sL(\eta)(\alpha) = \sL(\eta)(\sL(\rho)(\varphi)) = \sL(\pi)(\sL(\xi)(\varphi)) \in \sL(\pi)(\cC),
\]
so $\sL(\rho)(\sL(\xi)^{-1}(\cC)) \subset \sL(\eta)^{-1}(\sL(\pi)(\cC))$.

For the other direction, let $\alpha \in (\sL(\eta)^{-1}(\sL(\pi)(\cC)))(k')$ and by slight abuse of notation, also let $\alpha$ denote the corresponding map $\Spec(k'\llbracket t \rrbracket) \to \cW$. Possibly by extending $k'$, we may assume there exists $\psi \in \cC(k')$ such that $\sL(\pi)(\psi) \cong \sL(\eta)(\alpha)$. Again by slight abuse of notation, also let $\psi$ denote the corresponding map $\Spec(k'\llbracket t \rrbracket) \to \cX$, so $\pi \circ \psi \cong \eta \circ \alpha$. Thus there exists $\varphi: \Spec(k' \llbracket t \rrbracket) \to \cZ$ such that $\xi \circ \varphi = \psi$ and $\rho \circ \varphi = \alpha$. By slight abuse of notation, let $\varphi$ also denote the corresponding object of $(\sL(\cZ))(k')$, so $\sL(\xi)(\varphi) = \psi$ and $\sL(\rho)(\varphi) = \alpha$. Therefore $\varphi \in (\sL(\xi)^{-1}(\cC))(k')$, so $\alpha \in (\sL(\rho)(\sL(\xi)^{-1}(\cC)))(k')$ and we are done.
\end{proof}

We now prove a liftability lemma for elements in the truncation of a cylinder.

\begin{lemma}
\label{lemmaLiftJetToArcSameField}
Let $\cX$ be an equidimensional finite type Artin stack over $k$ with separated diagonal and affine geometric stabilizers, let $\cU$ be a smooth open substack of $\cX$, and let $\cC \subset |\sL(\cX)|$ be a cylinder that is disjoint from $|\sL(\cX \setminus \cU)|$. There exists some $N_\cC$ such that for any $n \geq N_{\cC}$, any field extension $k'$ of $k$, and any $\varphi_n \in (\theta_n(\cC))(k')$, there exists some $\varphi \in \cC(k')$ such that $\theta_n(\varphi) \cong \varphi_n$.
\end{lemma}

\begin{proof}
Because $\cC$ is a cylinder, there exists some $N_0 \in \Z_{\geq 0}$ such that $\cC$ is the preimage along $\theta_{N_0}$ of some constructible subset of $\sL_{N_0}(\cC)$. For any $n \geq N_0$,
\[
	\cC = \theta_n^{-1}(\theta_n(\cC)).
\]
Let $\xi: X \to \cX$ be a smooth cover by a finite type $k$-scheme $X$ such that for all field extensions $k'$ of $k$, the map $X(k') \to \overline{\cX}(k')$ is surjective. Such a smooth cover exists, e.g, by \cite[Theorem 1.2(b)]{Deshmukh}, and by \cite[Lemmas 0DRP and 0DRQ]{stacks-project} we may assume that $X$ is equidimensional. Let $C = \sL(\xi)^{-1}(\cC) \subset \sL(X)$. Noting that $\xi^{-1}(\cU)$ is smooth and that $C$ is disjoint from $\sL(X) \setminus \sL(\xi^{-1}(\cU))$, \cite[Lemma 7.5]{SatrianoUsatine1} implies there exists some $N_C$ such that for any $n \geq N_{\cC}$, any field extension $k'$ of $k$, and any $\psi_n \in (\theta_n(C))(k')$, there exists some $\psi \in C(k')$ such that $\theta_n(\psi) = \psi_n$. Note that the proof of \cite[Lemma 7.5]{SatrianoUsatine1} still holds when the irreducibility hypothesis is replaced with equidimensionality. Set $N_\cC = \max(N_0, N_C)$.

Let $n \geq N_\cC$, let $k'$ be a field extension of $k$, and let $\varphi_n \in (\theta_n(\cC))(k')$. Because $X(k') \to \overline{\cX}(k')$ is surjective and $X \to \cX$ is smooth, there exists some $\psi_n \in \sL_n(X)(k')$ such that $\sL_n(\xi)(\psi_n) \cong \varphi_n$. Because $\varphi_n \in (\theta_n(\cC))(k')$, there exists some field extension $k''$ of $k'$ and some $\varphi' \in \cC(k'')$ such that $\theta_n(\varphi') \cong \varphi_n \otimes_{k'} k''$. Because $X \to \cX$ is smooth, there exists some $\psi' \in \sL(X)(k'')$ such that $\sL(\xi)(\psi') \cong \varphi'$ and $\theta_n(\psi') \cong \psi_n \otimes_{k'} k''$. Thus $\psi_n \in (\theta_n(C))(k')$, so by our choice of $N_C$, there exists some $\psi \in C(k')$ such that $\theta_n(\psi) = \psi_n$. Set $\varphi = \sL(\xi)(\psi) \in \sL(\cX)(k')$. Then $\theta_n(\varphi) \cong \varphi_n$. Because $n \geq N_0$, we have $\cC = \theta_n^{-1}(\theta_n(\cC))$. Thus $\varphi \in \cC(k')$, and we are done.
\end{proof}

We end this section with the following general result.

\begin{lemma}
\label{lemmaNotArcOfBoundaryImpliesTruncationInClosure}
Let $\cX$ be a locally finite type Artin stack over $k$. If $\cU$ is an open substack of $\cX$, then $\theta_0(|\sL(\cX)| \setminus |\sL(\cX \setminus \cU)|)$ is contained in the closure of $|\cU|$ in $|\cX|$. Furthermore, if $\cW$ is a closed substack of $\cX$ such that $|\cU| \subset |\cW|$, then
\[
	|\sL(\cX)| \setminus |\sL(\cX \setminus \cU)| \subset |\sL(\cW)|.
\]
\end{lemma}

\begin{proof}
Let $k'$ be a field extension of $k$, and let $\varphi: \Spec(k'\llbracket t \rrbracket) \to \cX$ correspond to a $k'$-point of $\sL(\cX)$ that is not in $\sL(\cX \setminus \cU)$. Then the generic point $\Spec(k'\llparenthesis t \rrparenthesis) \to \Spec(k' \llbracket t \rrbracket) \xrightarrow{\varphi} \cX$ factors through $\cU$, so the special point $\theta_0(\varphi): \Spec(k') \hookrightarrow \Spec(k\llbracket t \rrbracket) \xrightarrow{\varphi} \cX$ has image in the closure of $|\cU|$ in $|\cX|$, verifying the first sentence of the proposition. We have also shown
\[
	\varphi(|\Spec(k'\llbracket t \rrbracket)|) \subset |\cW|.
\]
Then because $\Spec(k'\llbracket t \rrbracket)$ is reduced, $\varphi$ factors through $\cW$, so we have verified the second sentence of the proposition.
\end{proof}

\section{Motivic change of variables for representable resolutions of singularities}

We will ultimately prove Theorem \ref{theoremMotivicChangeOfVariablesMeasurable} by reducing to the case of smooth stacks and making use of our motivic change of variables formula \cite[Theorem 1.3]{SatrianoUsatine2}. The first step in this reduction is to prove the following change of variables formula for representable resolutions of singularities.

\begin{theorem}
\label{theoremRepresentableChangeOfVariables}
Let $\cX$ be a locally finite type Artin stack over $k$ with affine geometric stabilizers and separated diagonal. Let $\cZ$ be an irreducible smooth finite type Artin stack over $k$, and let $\pi: \cZ \to \cX$ be a morphism. Assume that $\pi$ factors as a proper morphism that is representable by schemes followed by an open immersion. Let $\cU$ be an open substack of $\cX$ such that $\pi^{-1}(\cU) \to \cU$ is an isomorphism, let $\cC \subset |\sL(\cX)|$ be a cylinder that is disjoint from $|\sL(\cX \setminus \cU)|$, and assume that $\theta_0(\cC) \subset \pi(|\cZ|)$. Set $\cE = \sL(\pi)^{-1}(\cC)$.

\begin{enumerate}[(a)]

\item\label{theoremPartCylinderBijection} The subset $\cE \subset |\sL(\cZ)|$ is a cylinder, and for any field extension $k'$ of $k$, the map $\overline{\cE}(k') \to \overline{\cC}(k')$ induced by $\pi$ is a bijection.

\item\label{theoremPartConstructibleRepresentableOfMCVF} The restriction of $\het^{(0)}_{L_{\cZ/\cX}}: |\sL(\cZ)| \to \Z \cup \{\infty\}$ to $\cE$ is integer valued and constructible.

\item\label{theoremPartIntegralOfMCVF} There exists $N \in \Z_{\geq 0}$ such that for all $n \in \Z_{\geq N}$,
\[
	\e(\theta_n(\cC))\bL^{-(n+1)\dim\cZ} = \int_{\cE} \bL^{-\het^{(0)}_{L_{\cZ/\cX}}}\diff\mu_\cZ.
\]

\end{enumerate}
\end{theorem}

\begin{remark}
Because $\pi: \cZ \to \cX$ is proper, $\cZ$ has separated diagonal and thus its geometric stabilizers are schemes by \cite[Tag 0B8D]{stacks-project}. Therefore because $\pi: \cZ \to \cX$ is representable, $\cZ$ has affine geometric stabilizers, so the motivic measure $\mu_\cZ$ is well defined \cite[Definition 2.3]{SatrianoUsatine2}. Furthermore, part (\ref{theoremPartConstructibleRepresentableOfMCVF}) implies that the integral $\int_{\cE} \bL^{-\het^{(0)}_{L_{\cZ/\cX}}}\diff\mu_\cZ$ is well defined \cite[Definition 2.8]{SatrianoUsatine2}. Also note that $\pi(|\Z|)$ is quasi-compact, so $\cC$ is small and thus each $\theta_n(\cC)$ is a quasi-compact locally constructible subset of $|\sL_n(\cX)|$ by \autoref{propositionSmallCylinderHasQuasiCompactConstructibleImages}. In particular $\e(\theta_n(\cC))$ is well defined.
\end{remark}

The following result gives a representable resolution of singularities for stacks. This will be used throughout this paper.

\begin{proposition}
\label{propositionResOfSingForStacks}
Let $\cW$ be a reduced locally finite type Artin stack over $k$, and let $\cU$ be the locus of $\cW$ that is smooth over $k$. Then there exists a smooth Artin stack $\cZ$ over $k$ and a morphism $\pi: \cZ \to \cW$ that is proper and representable by schemes such that $\pi^{-1}(\cU) \to \cU$ is an isomorphism and $|\pi^{-1}(\cU)|$ is dense in $|\cZ|$. In particular, if $\cW$ is irreducible then $\cZ$ is irreducible, and if $\cW$ is finite type over $k$ then $\cZ$ is finite type over $k$.
\end{proposition}

\begin{proof}
By considering a smooth cover of $\cW$ by a scheme, this is a straightforward consequence of functorial resolution of singularities, i.e., a resolution algorithm that is functorial with respect to smooth morphisms for the case where $\cW$ is a scheme. Such algorithms exist, e.g., by \cite[Theorem 8.1.2]{AbramovichTemkinWlodarczyk} or see \cite[Theorem 1.1.6]{Wlodarczyk} for the not-necessarily-equidimensional case. Note also that a morphism being representable by schemes is a property that is smooth local on the target.
\end{proof}

Our next goal is to prove that for our cylinders of interest, images of small cylinders remain small. We do so after a preliminary technical lemma.

\begin{lemma}
\label{lemmaTruncationInTargetTruncationInSource}
Let $\cX$ be a smooth finite type Artin stack over $k$, let $Y$ be a finite type scheme over $k$, let $\pi: \cX \to Y$ be a morphism, let $\cU$ be a smooth open substack of $\cX$ such that $\cU \hookrightarrow \cX \xrightarrow{\pi} \cY$ is an open immersion, let $\cC \subset |\sL(\cX)|$ be a cylinder that is disjoint from $|\sL(\cX \setminus \cU)|$, let $\ell \in \Z$ be such that $\cC$ is the preimage along $\theta_\ell$ of a constructible subset of $|\sL_\ell(\cX)|$, suppose there exists $h \in \Z$ such that $\het^{(0)}_{L_{\cX/Y}}$ is equal to $h$ on all of $\cC$, let $\sJ$ be the $(\dim Y)$th Fitting ideal of $\Omega^1_Y$, suppose there exists $e \in \Z$ such that $\ord_{\sJ} \circ \sL(\pi)$ is equal to $e$ on all of $\cC$, and assume that $\cX$ and $Y$ are equidimensional.

Let $k'$ be a field extension of $k$, let $\varphi \in \cC(k')$, let $\alpha \in \sL(Y)(k')$, and let $n \in \Z_{\geq 0}$.

If $n \geq \max(\ell + h, e)$ and $\theta_n(\sL(\pi)(\varphi)) = \theta_n(\alpha)$, then there exists some $\varphi' \in \cC(k')$ such that $\sL(\pi)(\varphi') = \alpha$ and $\theta_{n-h}(\varphi) \cong \theta_{n-h}(\varphi')$.
\end{lemma}

\begin{proof}
Set $\varphi^{(0)} = \varphi$. We will find a sequence $\{\varphi^{(i)}\}_{i \in \Z_{\geq 0}} \subset \cC(k')$ such that for all $i \in \Z_{\geq 0}$, we have $\theta_{n-h+i}(\varphi^{(i+1)}) \cong \theta_{n-h+i}(\varphi^{(i)})$ and $\theta_{n+i}(\sL(\pi)(\varphi^{(i)})) = \theta_{n+i}(\alpha)$. We will do so inductively on $i$, so assume we already have $\varphi^{(0)}, \dots, \varphi^{(i)}$. We know $\varphi^{(i)} \in \cC(k')$, $n + i \geq \max(\ell + h, e)$, and $\theta_{n+i}(\sL(\pi)(\varphi^{(i)})) = \theta_{n+i}(\alpha)$, so \cite[Proposition 5.2]{SatrianoUsatine2} implies that there exists $\varphi^{(i+1)} \in \sL(\cX)(k')$ such that $\theta_{n-h+i}(\varphi^{(i+1)}) \cong \theta_{n-h+i}(\varphi^{(i)})$ and $\theta_{n+i+1}(\sL(\pi)(\varphi^{(i+1)})) = \theta_{n+i+1}(\alpha)$ (for more details see the paragraph following \cite[Proposition 5.2]{SatrianoUsatine2}, and note that the proof of \cite[Proposition 5.2]{SatrianoUsatine2} never uses the injectivity hypothesis on $\cC(k')$ and still works if the hypothesis that $\cX$ and $Y$ are irreducible is replaced with the assumption that they are equidimensional). Because $n-h+i \geq \ell$, we also have $\varphi^{(i+1)} \in \cC(k')$, so we have the desired sequence.

Because $\theta_{n-h+i}(\varphi^{(i+1)}) \cong \theta_{n-h+i}(\varphi^{(i)})$ for all $i \in \Z_{\geq 0}$, there exists some $\varphi' \in \sL(\cX)(k')$ such that $\theta_{n-h+i}(\varphi') \cong \theta_{n-h+i}(\varphi^{(i+1)})$ for all $i \in \Z_{\geq 0}$. Thus for all $i \in \Z_{\geq 0}$,
\begin{align*}
	\theta_{n-h+i}(\sL(\pi)(\varphi')) &= \sL_{n-h+i}(\pi)(\theta_{n-h+i}(\varphi')) \\
	&= \sL_{n-h+i}(\pi)(\theta_{n-h+i}(\varphi^{(i+1)})) \\
	&= \theta_{n-h+i}(\sL(\pi)(\varphi^{(i+1)})) \\
	&= \theta_{n-h+i}(\alpha),
\end{align*}
so $\sL(\pi)(\varphi') = \alpha$. Also
\[
	\theta_{n-h}(\varphi') \cong \theta_{n-h}(\varphi^{(1)}) \cong \theta_{n-h}(\varphi^{(0)}) = \theta_{n-h}(\varphi).
\]
Because $n - h \geq \ell$, this also implies $\varphi' \in \cC(k')$, and we are done.
\end{proof}

\begin{proposition}
\label{propositionImageOfCylinderIsCylinder}
Let $\cX$ and $\cY$ be locally finite type Artin stacks over $k$, let $\pi: \cX \to \cY$ be a morphism, let $\cU$ be a smooth open substack of $\cX$ such that $\cU \hookrightarrow \cX \xrightarrow{\pi} \cY$ is an open immersion, let $\cC \subset |\sL(\cX)|$ be a small cylinder that is disjoint from $|\sL(\cX \setminus \cU)|$, and assume that $\cY$ is equidimensional. Then $\sL(\pi)(\cC) \subset |\sL(\cY)|$ is a small cylinder.
\end{proposition}

\begin{proof}
By replacing $\cX$ with a quasi-compact open substack $\cX_1$ that contains $\theta_0(\cC)$, replacing $\cU$ with $\cU \cap \cX_1$, and replacing $\cY$ with a quasi-compact open substack $\cY_1$ such that $\pi(|\cX_1|) \subset |\cY_1|$, we may assume that $\cX$ and $\cY$ are finite type over $k$.

Let $\eta: W \to \cY$ be a smooth cover by a finite type $k$-scheme $W$. Because $\cY$ is equidimensional and $W \to \cY$ is smooth, every connected component of $W$ is equidimensional, e.g., by \cite[Lemmas 0DRP and 0DRQ]{stacks-project}. Therefore, possibly replacing some of these connected components with their products with an affine space, we may assume that $W$ is equidimensional. Let $\rho: \cZ \to W$ be the base change of $\pi$ along $W \xrightarrow{\eta} \cY$, and let $\cE$ be the preimage of $\cC$ in $|\sL(\cZ)|$. By \autoref{lemmaPreimageImageImagePreimage},
\[
	\sL(\eta)^{-1}(\sL(\pi)(\cC)) = \sL(\rho)(\cE),
\]
so by \autoref{lemmaCylinderCanBeCheckedInSmoothCover}, we may replace $\pi, \cX, \cY, \cU, \cC$ with $\rho, \cZ, W, \cU \times_\cY W, \cE$ and thus assume that $\cY$ is a scheme $Y$.

By replacing $\pi: \cX \to Y$ with its reduction, we may assume that $\cX$ and $Y$ are reduced. Then by \autoref{propositionResOfSingForStacks}, there exists a smooth finite type Artin stack $\cX_2$ over $k$ and a morphism $\xi: \cX_2 \to \cX$ that is proper and representable by schemes such that $\xi^{-1}(\cU) \to \cU$ is an isomorphism. By replacing $\cX_2$ with the union of its connected components that intersect $\xi^{-1}(\cU)$, we may assume that $\cX_2$ is equidimensional. Because $\xi$ is proper and representable by schemes and an isomorphism above $\cU$, the induced map 
\[
	|\sL(\cX_2)| \setminus |\sL(\cX_2 \setminus \xi^{-1}(\cU))| \to |\sL(\cX)| \setminus |\sL(\cX \setminus \cU)|
\]
is surjective. In particular $\sL(\xi)(\sL(\xi^{-1})(\cC)) = \cC$. Therefore by replacing $\pi, \cX, \cU, \cC$ with $\pi \circ \xi, \cX_2, \xi^{-1}(\cU), \sL(\xi)^{-1}(\cC)$, we may assume that $\cX$ is smooth and equidimensional.

Let $\sJ$ be the $(\dim Y)$th Fitting ideal of $\Omega^1_Y$. Because $\cU \hookrightarrow \cX \xrightarrow{\pi} \cY$ is an open immersion and $\cU$ is smooth, we have that $\ord_{\sJ} \circ \sL(\pi)$ and $\het^{(0)}_{L_{\cX/Y}}$ are constructible functions on $\cC$ \cite[Remarks 2.6 and 3.5]{SatrianoUsatine2}. These constructible functions take only finitely many values on $\cC$ \cite[Proposition 2.2]{SatrianoUsatine2}, so by separately considering each stratum of $\cC$ where $\ord_{\sJ} \circ \sL(\pi)$ and $\het^{(0)}_{L_{\cX/Y}}$ are both constant, we may assume there exist $h, e \in \Z$ such that $\het^{(0)}_{L_{\cX/Y}}$ is equal to $h$ on all of $\cC$ and $\ord_{\sJ} \circ \sL(\pi)$ is equal to $e$ on all of $\cC$.

Because $\cC$ is a cylinder, there exists $\ell \in \Z_{\geq 0}$ and a constructible subset $\cC_\ell \subset |\sL_\ell(\cX)|$ such that $\cC = \theta_\ell^{-1}(\cC_\ell)$. Set $\cC_{\ell + h+e} = \theta_{\ell + h+e}(\cC) \subset |\sL_{\ell + h+e}(\cC)|$. By Chevalley's Theorem for Artin stacks \cite[Theorem 5.1]{HallRydh17}, $\sL_{\ell + h+e}(\pi)(\cC_{\ell + h+e})$ is a constructible subset of $\sL_{\ell+h+e}(Y)$. We are thus done if we can show
\[
	\sL(\pi)(\cC) = \theta_{\ell+h+e}^{-1}(\sL_{\ell + h+e}(\pi)(\cC_{\ell + h+e})).
\]
Clearly $\sL(\pi)(\cC) \subset \theta_{\ell+h+e}^{-1}(\sL_{\ell + h+e}(\pi)(\cC_{\ell + h+e}))$. 

To show the other inclusion, let $\alpha \in (\theta_{\ell+h+e}^{-1}(\sL_{\ell + h+e}(\pi)(\cC_{\ell + h+e})))(k')$ for some field extension $k'$ of $k$. Possibly after extending $k'$, we may assume there exists some $\varphi \in \cC(k')$ such that $\theta_{\ell+h+e}(\sL(\pi)(\varphi)) = \theta_{\ell+h+e}(\alpha)$. Then by \autoref{lemmaTruncationInTargetTruncationInSource} there exists some $\varphi' \in \cC(k')$ such that $\sL(\pi)(\varphi') = \alpha$, so $\alpha \in (\sL(\pi)(\cC))(k')$, completing our proof.
\end{proof}

We turn now to the main theorem of this section.

\begin{proof}[{Proof of Theorem \ref{theoremRepresentableChangeOfVariables}}]
By assumption $\pi$ is the composition of a proper morphism $\cZ \to \cX_1$ that is representable by schemes and an open immersion $\cX_1 \to \cX$. Let $\cX_2$ be a quasi-compact open substack of $\cX$ such that $\pi(|\cZ|) \subset |\cX_2|$. Replacing $\cX$ with $\cX_1 \cap \cX_2$, we may assume that $\cX$ is finite type over $k$ and that $\pi$ is a proper morphism that is representable by schemes. Because the theorem is obvious when $\cU = \emptyset$, we may also assume that $\cU$ is nonempty. In particular, $|\pi^{-1}(\cU)|$ is dense in $|\cZ|$.

Let $\cW$ be the scheme theoretic image of the inclusion $\cU \to \cX$. Because $|\pi^{-1}(\cU)|$ is dense in $|\cZ|$, we have $|\pi(\cZ)| \subset |\cW|$. Then because $\cZ$ is reduced, $\pi$ factors as $\cZ \xrightarrow{\rho} \cW \hookrightarrow \cX$. We will show that we can replace $\cX, \pi$ with $\cW, \rho$ and therefore assume that $|\cU|$ is dense in $|\cX|$. First, note that because $\pi$ is proper and representable by schemes and $\cW \hookrightarrow \cX$ is a closed immersion, the morphism $\rho$ is proper and representable by schemes. By \autoref{lemmaNotArcOfBoundaryImpliesTruncationInClosure}, we have $\cC \subset |\sL(\cW)|$. In particular, each $\theta_n(\cC) \subset |\sL_n(\cW)| \subset |\sL_n(\cX)|$, so each class $\e(\theta_n(\cC))$ does not depend on whether we consider $\theta_n(\cC)$ as a constructible subset of $|\sL_n(\cW)|$ or $|\sL_n(\cX)|$. Thus to show that we may replace $\cX, \pi$ with $\cW, \rho$, we just need to show that
\[
	\het^{(0)}_{L_{\cZ/\cX}} = \het^{(0)}_{L_{\cZ/\cW}}.
\]
To do so, consider an arc $\varphi: \Spec(k'\llbracket t \rrbracket) \to \cZ$. The exact triangle
\[
	L\rho^* L_{\cW/\cX} \to L_{\cZ / \cX} \to L_{\cZ / \cW}
\]
induces an exact triangle
\[
	L\varphi^*L\rho^* L_{\cW/\cX} \to L\varphi^*L_{\cZ / \cX} \to L\varphi^*L_{\cZ / \cW},
\]
so we obtain the exact sequence
\[
	H^0(L\varphi^*L\rho^* L_{\cW/\cX}) \to H^0(L\varphi^*L_{\cZ / \cX}) \to H^0(L\varphi^*L_{\cZ / \cW}) \to H^1(L\varphi^*L\rho^* L_{\cW/\cX}).
\]
Because $\cW \to \cX$ is a closed immersion, 
\[
	H^0(L\varphi^*L\rho^* L_{\cW/\cX}) = H^1(L\varphi^*L\rho^* L_{\cW/\cX}) = 0,
\]
and therefore
\[
	\het^{(0)}_{L_{\cZ/\cX}}(\varphi) = \dim_{k'} H^0(L\varphi^*L_{\cZ / \cX}) = \dim_{k'} H^0(L\varphi^*L_{\cZ / \cW}) = \het^{(0)}_{L_{\cZ/\cW}}(\varphi).
\]
Thus we may replace $\cX, \pi$ with $\cW, \rho$, so we may assume that $|\cU|$ is dense in $|\cX|$. 

We will now prove each part of the theorem.

\begin{enumerate}[(a)]

\item Because $\cC$ is a cylinder and $\cX$ is finite type over $k$, there exists some $n \in \Z_{\geq 0}$ and a constructible subset $\cC_n \subset |\sL_n(\cX)|$ such that $\cC = \theta_n^{-1}(\cC_n)$. Therefore $\cE = \theta_n^{-1}(\sL_n(\pi)^{-1}(\cC_n))$ is a cylinder. Because $\cC$ is disjoint from $|\sL(\cX \setminus \cU)|$, the valuative criterion for proper morphisms that are representable by schemes implies that $\overline{\cE}(k') \to \overline{\cC}(k')$ is bijective for all field extensions $k'$ of $k$.

\item Noting that $\pi$ being representable implies $\het^{(1)}_{L_{\cZ/\cX}} = 0$, we get that the function $\het^{(0)}_{L_{\cZ/\cX}}$ is integer valued and constructible on $\cE$ by \cite[Remark 3.5]{SatrianoUsatine2}.

\item For each $h \in \Z_{\geq 0}$, let $\cE^{(h)} = \cE \cap (\het^{(0)}_{L_{\cZ/\cX}})^{-1}(h)$ and $\cC^{(h)} = \sL(\pi)(\cE^{(h)})$. By part (\ref{theoremPartConstructibleRepresentableOfMCVF}), $\cE = \bigsqcup_{h \in \Z_{\geq 0}} \cE^{(h)}$, so by part (\ref{theoremPartCylinderBijection}), $\cC = \bigsqcup_{h \in \Z_{\geq 0}} \cC^{(h)}$ and $\cE^{(h)} = \sL(\pi)^{-1}(\cC^{(h)})$ for all $h$. By parts (\ref{theoremPartCylinderBijection}) and (\ref{theoremPartConstructibleRepresentableOfMCVF}) each $\cE^{(h)}$ is a cylinder. Because $\cZ$ is irreducible and $|\cU|$ is dense in $|\cX|$, we have $\cX$ is irreducible and therefore equidimensional \cite[Lemma 0DRX]{stacks-project}. Thus by \autoref{propositionImageOfCylinderIsCylinder}, each $\cC^{(h)}$ is a cylinder. By \cite[Proposition 2.2]{SatrianoUsatine2}, $\cE^{(h)}$ is empty for all but finitely many $h$, so
\[
	\int_{\cE} \bL^{-\het^{(0)}_{L_{\cZ/\cX}}}\diff\mu_\cZ = \sum_{h \in \Z_{\geq 0}} \int_{\cE^{(h)}} \bL^{-\het^{(0)}_{L_{\cZ/\cX}}}\diff\mu_\cZ,
\]
with the right-hand-side being a finite sum. Because $\cC^{(h)}$ is empty for all but finitely many $h$, there exists some $N_0 \in \Z_{\geq 0}$ such that for all $h$, we have $\cC^{(h)}$ is the preimage along $\theta_{N_0}$ of a constructible subset of $|\sL_{N_0}(\cX)|$. Then for all $n \geq N_0$,
\[
	\theta_n(\cC) = \bigsqcup_{h \in \Z_{\geq 0}} \theta_n(\cC^{(h)}),
\]
so
\[
	\e(\theta_n(\cC))\bL^{-(n+1)\dim\cZ} = \sum_{h \in \Z_{\geq 0}} \e(\theta_n(\cC^{(h)}))\bL^{-(n+1)\dim\cZ},
\]
with the right-hand-side being a finite sum. Therefore by considering each $\cC^{(h)}, \cE^{(h)}$ in place of $\cC, \cE$, we may assume that there exists some $h \in \Z_{\geq 0}$ such that $\het^{(0)}_{L_{\cZ/\cX}}$ is equal to $h$ on all of $\cE$.

Because $\cZ$ is a smooth and irreducible and $\cE$ is a cylinder, there exists some $N_1 \in \Z_{\geq 0}$ such that for all $n \geq N_1$,
\[
	\int_{\cE} \bL^{-\het^{(0)}_{L_{\cZ/\cX}}}\diff\mu_\cZ = \bL^{-h} \mu_\cZ(\cE) = \bL^{-h}\e(\theta_n(\cE))\bL^{-(n+1)\dim\cZ}.
\]
Thus it is sufficient to show that there exists some $N_2 \in \Z_{\geq 0}$ such that for all $n \geq N_2$,
\[
	\e(\theta_n(\cE)) = \bL^h \e(\theta_n(\cC)).
\]
Using that $\cZ$ is smooth and that $\cC$ is a cylinder, it is straightforward to check that there exists some $N_3 \in \Z_{\geq 0}$ such that for all $n \geq N_3$,
\[
	\sL_n(\pi)^{-1}(\theta_n(\cC)) \subset \theta_n(\cE).
\]
Therefore by \cite{SatrianoUsatine1} it is sufficient to show that there exists some $N_4 \in \Z_{\geq 0}$ such that for any $n \geq N_4$, any field extension $k'$ of $k$, and any $\varphi_n \in (\theta_n(\cC))(k')$,
\[
	(\sL_n(\pi)^{-1}(\varphi_n))_\red \cong \bA_{k'}^h.
\]
Let $\xi: X \to \cX$ be a smooth cover by a finite type $k$-scheme $X$ such that for all field extensions $k'$ of $k$, the map $X(k') \to \overline{\cX}(k')$ is surjective. Such a smooth cover exists, e.g, by \cite[Theorem 1.2(b)]{Deshmukh}, and by \cite[Lemmas 0DRP and 0DRQ]{stacks-project} we may assume that $X$ is equidimensional. Let $\rho: Z \to X$ be the base change of $\pi: \cZ \to \cX$ along $\xi: X \to \cX$. Because $\pi$ is representable by schemes, $Z$ is a scheme. Let $U = \xi^{-1}(\cU)$. Because $Z \to |\cZ|$ is an open map and $|\pi^{-1}(\cU)|$ is dense in $|\cZ|$, we have that $\rho^{-1}(U)$ is dense in $Z$. Then because $\rho^{-1}(U) \to U$ is an isomorphism, $Z$ is equidimensional. Let $C = \sL(\xi)^{-1}(\cC) \subset \sL(X)$, and let $E = \sL(\rho)^{-1}(C) \subset \sL(Z)$. Clearly $C$ and $E$ are cylinders, and by the valuative criterion applied to the proper map $\rho:Z \to X$, we have $E(k') \to C(k')$ is bijective for all field extensions $k'$ of $k$. Furthermore, because $X \to \cX$ is flat, $L_{Z/X}$ is isomorphic to the derived pullback of $L_{\cZ/\cX}$ along $Z \to \cZ$. In particular, $\het^{(0)}_{L_{Z/X}}$ is equal to $h$ on all of $E$. By \cite[Chapter 5 Theorem 3.2.2]{ChambertLoirNicaiseSebag}, noting that $\het^{(0)}_{L_{Z/X}}$ coincides with the function denoted $\ordjac_{\rho}$ in \cite{ChambertLoirNicaiseSebag}, there exists some $N_5 \in \Z_{\geq 0}$ such that for any $n \geq N_5$, any field extension $k'$ of $k$, and any $\psi_n \in (\theta_n(C))(k')$,
\[
	(\sL_n(\rho)^{-1}(\psi_n))_\red \cong \bA_{k'}^h.
\]
Let $N_\cC$ be as in the conclusion of \autoref{lemmaLiftJetToArcSameField}, and set $N_4 = \max(N_5, N_{\cC})$.

Now let $n \geq N_4$, let $k'$ be a field extension of $k$, and let $\varphi_n \in (\theta_n(\cC))(k')$. By our choice of $N_\cC$, there exists $\varphi \in \cC(k')$ such that $\theta_n(\varphi) \cong \varphi_n$. Because $X(k') \to \overline{\cX}(k')$ is surjective and $X \to \cX$ is smooth, there exists some $\psi \in C(k')$ such that $\sL(\xi)(\psi) \cong \varphi$. Then
\[
	(\sL_n(\pi)^{-1}(\varphi_n))_\red \cong (\sL_n(\rho)^{-1}(\theta_n(\psi)))_\red \cong \bA_{k'}^h,
\]
completing the proof.\qedhere
\end{enumerate}
\end{proof}

\section{Stacks and the Grothendieck group of varieties}

We prove $\Vert \e(\cY) \Vert = \exp(\dim\cY)$ for our stacks of interest. We make use of this in the following section to prove properties of our measure $\mu_{\cX,d}$.
%the value of the non-Archimedean semi-norm 

\begin{lemma}
\label{lemmaDimensionAndStratification}
Let $\cY$ be a finite type Artin stack over $k$, and let $\{\cY_i\}_{i \in I}$ be a finite collection of locally closed substacks of $\cY$ such that $|\cY|=\coprod_{i\in I}|\cY_i|$. Then
\[
	\dim\cY = \max_{i \in I} \dim \cY_i.
\]
\end{lemma}

\begin{proof}
Let $Y\to\cY$ be a smooth cover by a scheme, let $R=Y\times_\cY Y$, let $s,t\colon R\to Y$ be the two projection maps, and let $e\colon Y\to R$ be the canonical section. Let $Y_i:=Y\times_\cY \cY_i$, let $R_i:=R\times_\cY\cY_i$, and let $s_i,t_i\colon R_i\to Y_i$ and $e_i\colon Y_i\to R_i$ be the induced maps.

Fix $i$, let $y'_i\in Y_i$, and let $y_i\in|\cY|$ be the image of $y'_i$. Since $Y_i\subset Y$ is locally closed, we have $\dim_{y'_i} Y_i\leq \dim_{y'_i} Y$ with equality if $Y_i$ is open in $Y$, i.e., if $\cY_i\subset\cY$ is open. Also note that $$(R_i)_{y'_i}:=R_i\times_{s_i,Y_i,y'_i} \Spec k(y'_i)=R\times_{s,Y,y'_i} \Spec k(y'_i)=:R_{y'_i}.$$ Then by definition,
\begin{align*}
	\dim_{y_i} \cY_i &= \dim_{y'_i} Y_i - \dim_{e(y'_i)} (R_i)_{y'_i}\\
	& \leq \dim_{y'_i} Y - \dim_{e(y'_i)} R_{y'_i} = \dim_{y_i} \cY
\end{align*}
with equality if $\cY_i$ is open in $\cY$. The result follows since $\dim\cY=\sup_{y\in|\cY|}\{\dim_y\cY\}$, $\dim\cY_i=\sup_{y\in|\cY_i|}\{\dim_y\cY_i\}$, and for every irreducible component $\cZ$ of $\cY$, there exists $i\in I$ with $\cZ\cap\cY_i$ open in $\cZ$.
\end{proof}

Let $\EP: \widehat\sM_k \to \Q\llparenthesis T^{-1} \rrparenthesis$ be the unique continuous ring homomorphism that takes each class $\e(Y)$ of a finite type $k$-scheme $Y$ to its Euler--Poincar\'{e} polynomial, and let $\val: \Q\llparenthesis T^{-1} \rrparenthesis \to \Z \cup \{\infty\}$ denote the valuation given by taking the order of vanishing of $T^{-1}$.

\begin{lemma}
\label{lemmaEulerPoincarePolynomialAndDimension}
Let $\cY$ be a finite type Artin stack over $k$ with affine geometric stabilizers. Then
\[
	\val(\EP(\e(\cY))) = -\dim\cY,
\]
and the coefficient of $T^{\dim\cY}$ in $\EP(\e(\cY))$ is positive.
\end{lemma}

\begin{remark}
Although later we will only use the first part of this lemma, the positivity part is useful in its proof. Specifically, this positivity allows us to avoid cancellation during the reduction steps below.
\end{remark}

\begin{proof}
By \autoref{lemmaDimensionAndStratification} and the fact that $\cY$ can be stratified into quotient stacks \cite[Proposition 3.5.9]{Kresch}, we may assume that $\cY = [Y / G]$, where $G$ is a general linear group acting on a finite type $k$-scheme $Y$. Note that
\[
	\dim[Y / G] = \dim Y - \dim G,
\]
and because $G$ is a special group, $\e(G)$ is a unit and
\[
	\e([Y / G]) = \e(Y) \e(G)^{-1}.
\]
Therefore we are done if the lemma holds in the case where $\cY$ is a scheme (applied to the schemes $Y$ and $G$), but this case is well known, e.g., by using Nagata compactifications, Chow's lemma, and resolution of singularities to reduce to the case of smooth and projective varieties where Euler--Poincar\'{e} polynomials coincide with Poincar\'{e} polynomials.
\end{proof}

\begin{proposition}
\label{propositionNormOfStackDimensionOfStack}
Let $\cY$ be a finite type Artin stack over $k$ with affine geometric stabilizers. Then 
\[
	\Vert \e(\cY) \Vert = \exp(\dim\cY).
\]
\end{proposition}

\begin{proof}
First we note that \autoref{lemmaEulerPoincarePolynomialAndDimension} implies $\Vert \e(\cY) \Vert \geq \exp(\dim\cY)$, so we only need to show $\Vert \e(\cY) \Vert \leq \exp(\dim\cY)$. Then by \autoref{lemmaDimensionAndStratification}, the fact that $\Vert \cdot \Vert$ satisfies the non-archimedean triangle inequality, and the fact that $\cY$ can be stratified into quotient stacks \cite[Proposition 3.5.9]{Kresch}, we may assume that $\cY = [Y / G]$, where $G$ is a general linear group acting on a finite type $k$-scheme $Y$.

Because $G$ is a special group, $\e(G)$ is a unit and 
\[
	\e(\cY) = \e(Y)\e(G)^{-1},
\]
and by \cite[Lemma 3.17]{SatrianoUsatine1},
\[
	\Vert \e(G)^{-1} \Vert = \Vert \e(G) \Vert^{-1}.
\]
Therefore
\[
	\Vert \e(\cY) \Vert \leq \Vert \e(Y) \Vert \cdot \Vert \e(G) \Vert^{-1} \leq \exp(\dim Y) \cdot \exp(\dim G)^{-1} = \exp(\dim \cY),
\]
where we note that $\Vert \e(Y) \Vert \leq \exp(\dim Y)$ by definition of $\Vert \cdot \Vert$ and $\Vert \e(G) \Vert \geq \exp(\dim G)$ by the beginning of this proof.
\end{proof}

The following consequence will be useful when studying the behaviour of $\mu_{\cX,d}$.

\begin{proposition}\label{propositionNormsAndSumsOfLimits}
Let $m \in \Z_{\geq 0}$ and for each $i \in \{1, \dots, m\}$, let $\{\cY_{i, n}\}_{n \in \Z_{\geq 0}}$ be a sequence of finite type Artin stacks over $k$ with affine geometric stabilizers such that $\{\e(\cY_{i,n})\}_{n \in \Z_{\geq 0}}$ converges as $n \to \infty$. Then
\[
	\Vert \sum_{i = 1}^m \lim_{n \to \infty}  \e(\cY_{i,n}) \Vert = \max_{1\leq i\leq m} \Vert \lim_{n \to \infty}  \e(\cY_{i,n}) \Vert.
\]
\end{proposition}

\begin{proof}
The proposition follows from the continuity of $\Vert \cdot \Vert: \widehat{\sM}_k \to \R_{\geq 0}$ and the fact that for all $n \in \Z_{\geq 0}$,
\begin{align*}
	\Vert \sum_{i=1}^m \e(\cY_{i, n}) \Vert &= \Vert \e( \cY_{1, n} \sqcup \dots \sqcup \cY_{m, n}) \Vert = \exp(\dim(\cY_{1, n} \sqcup \dots \sqcup \cY_{m, n}))\\
	&= \max_{1\leq i\leq m}\, \exp(\dim \cY_{i,n}) = \max_{1\leq i \leq m} \Vert \e(\cY_{i,n}) \Vert,
\end{align*}
where the second and final equalities are due to \autoref{propositionNormOfStackDimensionOfStack} and the third equality is due to \autoref{lemmaDimensionAndStratification} (or the definition of dimension).
\end{proof}

\begin{remark}\label{remarkMeasureWrittenAsEffectiveLimit}
Let $\cX$ be a locally finite type Artin stack over $k$ with affine geometric stabilizers, let $d \in \Z$, let $\cC \subset |\sL(\cX)|$ be a $d$-convergent cylinder, and let $m \in \Z$. Then there exists a sequence $\{\cY_n\}_{n \in \Z_{\geq 0}}$ of finite type Artin stacks over $k$ with affine geometric stabilizers such that $\bL^m\mu_{\cX, d}(\cC) = \lim_{n \to \infty} \e(\cY_n)$. Specifically, one may take $\cY_n=\bA^m\times\theta_n(\cC)\times B\bG_a^{(n+1)\dim\cX}$ where, if $m<0$, we define $\bA^m$ as $B\bG_a^{-m}$.
%This follows from \autoref{definitiondConvergent} and the fact that for all $n \in \Z_{\geq 0}$ we have $\bL^{-n} = \e(B(\bG_{a}^n))$ since $\bG_a$ is a special group.
\end{remark}

\section{$d$-convergence and $d$-measurable sets}

The following notion of convergence plays an integral role throughout the remainder of the paper.

\begin{definition}\label{definitiondConvergent}
Let $\cX$ be a locally finite type Artin stack over $k$ with affine geometric stabilizers, let $\cC \subset |\sL(\cX)|$ be a cylinder, and let $d \in \Z$. We call $\cC$ \emph{$d$-convergent} if $\cC$ is small and the sequence
\[
	\{\e(\theta_n(\cC)) \bL^{-(n+1)d}\}_{n \in \Z_{\geq 0}}
\]
converges. In that case, we set
\[
	\mu_{\cX, d}(\cC) = \lim_{n \to \infty} \e(\theta_n(\cC)) \bL^{-(n+1)d}.
\]
\end{definition}

\begin{remark}
If $\cX$ is equidimensional, finite type, and either smooth or a global quotient, then any cylinder $\cC \subset |\sL(\cX)|$ is $\dim \cX$-convergent by \cite[Theorems 3.9 and 3.33]{SatrianoUsatine1} and $\mu_{\cX,\dim \cX}(\cC) = \mu_\cX(\cC)$.
\end{remark}

\begin{proposition}
Let $\cX$ be a locally finite type Artin stack over $k$ with affine geometric stabilizers and separated diagonal, let $\cU$ be an open substack of $\cX$ that is smooth and finite type over $k$ and irreducible, and let $\cC \subset |\sL(\cX)|$ be a small cylinder that is disjoint from $|\sL(\cX \setminus \cU)|$. Then $\cC$ is $\dim\cU$-convergent.
\end{proposition}

\begin{proof}
By replacing $\cX$ with a quasi-compact open substack that contains $\theta_0(\cC)$ and replacing $\cU$ with its intersection with that quasi-compact open substack, we may assume that $\cX$ is finite type over $k$. Then the scheme theoretic image $\cX'$ of $\cU \hookrightarrow \cX$ is integral and finite type over $k$, so by \autoref{propositionResOfSingForStacks} there exists an irreducible smooth finite type Artin stack $\cZ$ over $k$ and a morphism $\pi': \cZ \to \cX'$ that is proper and representable by schemes and such that $\pi'^{-1}(\cU) \to \cU$ is an isomorphism. Considering the composition $\pi: \cZ \xrightarrow{\pi'} \cX' \hookrightarrow \cX$ and noting \autoref{lemmaNotArcOfBoundaryImpliesTruncationInClosure}, we see that $\cX, \cZ, \pi, \cU, \cC$ satisfy the hypotheses of \autoref{theoremRepresentableChangeOfVariables}. Therefore, noting that $\dim\cU = \dim\cZ$, \autoref{theoremRepresentableChangeOfVariables} parts (\ref{theoremPartConstructibleRepresentableOfMCVF}) and (\ref{theoremPartIntegralOfMCVF}) imply that $\cC$ is $\dim\cU$-convergent.
\end{proof}

Having defined $d$-convergence, we may now introduce the corresponding notion of $d$-measurability.

\begin{definition}\label{def:d-measurability-and-measure}
Let $\cX$ be a locally finite type Artin stack over $k$ with affine geometric stabilizers, and let $d \in \Z$. Let $\cC \subset |\sL(\cX)|$, let $\varepsilon \in \R_{>0}$, let $\cC^{(0)} \subset |\sL(\cX)|$ be a $d$-convergent cylinder, and let $\{\cC^{(i)}\}_{i \in I}$ be a collection of $d$-convergent cylinders in $|\sL(\cX)|$.

We call $(\cC^{(0)}, \{\cC^{(i)}\}_{i \in I})$ a \emph{$d$-convergent cylindrical $\varepsilon$-approximation} of $\cC$ if
\[
	(\cC \cup \cC^{(0)}) \setminus (\cC \cap \cC^{(0)}) \subset \bigcup_{i \in I} \cC^{(i)},
\]
and $\Vert \mu_{\cX,d}(\cC^{(i)}) \Vert < \varepsilon$ for all $i$.
\end{definition}

\begin{definition}\label{def:muxd}
Let $\cX$ be a locally finite type Artin stack over $k$ with affine geometric stabilizers, let $d \in \Z$, and let $\cC \subset |\sL(\cX)|$. We say $\cC$ is \emph{$d$-measurable} if for any $\varepsilon > 0$, the set $\cC$ has a $d$-convergent cylindrical $\varepsilon$-approximation. In that case we let $\mu_{\cX,d}(\cC) \in \widehat{\sM}_k$ denote the unique element of $\widehat{\sM}_k$ such that
\[
\Vert \mu_{\cX,d}(\cC) - \mu_{\cX,d}(\cC^{(0)}) \Vert < \varepsilon
\]
for any $d$-convergent cylindrical $\varepsilon$-approximation $(\cC^{(0)}, \{\cC^{(i)}\}_{i \in I})$ of $\cC$.
\end{definition}

\begin{remark}
When $\cC$ is $d$-measurable, such an element $\mu_{\cX, d}(\cC)$ exists by standard methods. See, e.g., \cite[Chapter 6, Theorem 3.3.2]{ChambertLoirNicaiseSebag} for the case where $\cX$ is a finite type equidimensional scheme and $d = \dim\cX$, and note that the exact same argument works in our setting.
\end{remark}

\begin{remark}
Suppose $\cX$ is equidimensional, finite type, and a global quotient. Then $\cC$ is $\dim\cC$-measurable if and only if $\cC$ is measurable in the sense of \cite[Definition 3.13]{SatrianoUsatine1}, and in that case $\mu_{\cX,\dim \cX}(\cC) = \mu_\cX(\cC)$. In particular, these definitions are consistent with the usual definitions in the case $\cX$ is an equidimensional and finite type scheme.
\end{remark}

Our notion of $d$-measurability allows us to speak of $d$-measurable functions.

\begin{definition}\label{definitionOfdMeasurable}
Let $\cX$ be a locally finite type Artin stack over $k$ with affine geometric stabilizers, let $d \in \Z$, let $\cC \subset |\sL(\cX)|$, and let $f: \cC \to \Z \cup \{\infty\}$. We say $f$ is $d$-measurable if for all $n \in \Z$, the set $f^{-1}(n)$ is a $d$-measurable subset of $|\sL(\cX)|$.
\end{definition}

\begin{definition}\label{definitionOfdIntegrable}
Let $\cX$ be a locally finite type Artin stack over $k$ with affine geometric stabilizers, let $d \in \Z$, let $\cC \subset |\sL(\cX)|$, and let $f: \cC \to \Z \cup \{\infty\}$. We say $\bL^{f}$ is $d$-integrable if $f$ is $d$-measurable and the series $\sum_{n \in \Z} \bL^n \mu_{\cX, d}(f^{-1}(n))$ converges. In that case we set
\[
	\int_\cC \bL^f \diff\mu_{\cX,d} = \sum_{n \in \Z} \bL^n \mu_{\cX, d}(f^{-1}(n)).
\]
\end{definition}

\begin{remark}
If $f$ is a function defined on a larger subset than $\cC$, we may say $f$ is $d$-measurable (resp. $\bL^{f}$ is $d$-integrable) on $\cC$ when we mean \autoref{definitionOfdMeasurable} (resp. \autoref{definitionOfdIntegrable}) is satisfied by the restriction of $f$ to $\cC$.
\end{remark}

\begin{notation}
Let $\cX$ be a locally finite type Artin stack over $k$ with affine geometric stabilizers, let $\cC \subset |\sL(\cX)|$, let $d \in \Z$, let $f: \cC \to \Z$ be a function that takes only finitely many values, and assume that for all $n \in \Z$ the set $f^{-1}(n)$ is a $d$-convergent cylinder in $|\sL(\cX)|$. Then we will set
\[
	\int_\cC \bL^f \diff\mu_{\cX, d} = \sum_{n \in \Z} \bL^n \mu_{\cX, d}( f^{-1}(n) ).
\]
\end{notation}

For the next two results, fix $\cX$ a locally finite type Artin stack over $k$ with affine geometric stabilizers, and fix $d \in \Z$. 

\begin{lemma}\label{lemmaNormInclusionCylinders}
Let $\cC, \cD \subset |\sL(\cX)|$ be $d$-convergent cylinders. If $\cC \subset \cD$, then
\[
	\Vert \mu_{\cX,d}(\cC) \Vert \leq \Vert \mu_{\cX,d}(\cD) \Vert.
\]
\end{lemma}

\begin{proof}
It is straightforward to check that $\cD \setminus \cC$ is a $d$-convergent cylinder and that $\mu_{\cX, d}(\cD) = \mu_{\cX,d}(\cC) + \mu_{\cX,d}(\cD \setminus \cC)$. The lemma then follows from \autoref{propositionNormsAndSumsOfLimits} and \autoref{remarkMeasureWrittenAsEffectiveLimit}.
\end{proof}

\begin{proposition}
\label{propositionConditionForUnionOfMeasurablesIsMeasurable}
Let $\{\cC^{(n)}\}_{n \in \Z_{\geq 0}}$ be a sequence of pairwise disjoint $d$-measurable subsets of $|\sL(\cX)|$. Then $\lim_{n \to \infty} \mu_{\cX, d}( \cC^{(n)}) = 0$ if and only if $\bigcup_{n \in \Z_{\geq 0}} \cC^{(n)}$ is $d$-measurable, and in that case,
\[
	\mu_{\cX,d}\left(\bigcup_{n \in \Z_{\geq 0}} \cC^{(n)}\right) = \sum_{n \in \Z_{\geq 0}} \mu_{\cX,d}(\cC^{(n)}).
\]
\end{proposition}

\begin{proof}
A version of this proposition including the special case where $\cX$ is a finite type scheme over $k$ is proved in \cite[Chapter 6 Proposition 3.4.3]{ChambertLoirNicaiseSebag}. Using \autoref{propositionSmallCylinderFiniteSubcover} and \autoref{lemmaNormInclusionCylinders}, the same proof works here.
\end{proof}

\section{Relative height functions}

As mentioned in the introduction, our motivic change of variables formula is controlled by the following relative height function which we introduce.

\begin{notation}
If $k'$ is a field extension of $k$ and $M$ is a module over $k'\llbracket t \rrbracket$, we will let $M_\tor$ denote the torsion submodule of $M$.
\end{notation}

\begin{definition}\label{def:relhet}
Let $\pi: \cX \to \cY$ be a morphism of locally finite type Artin stacks over $k$. Let $k'$ be a field extension of $k$, and let $\varphi \in \sL(\cX)(k')$. The canonical morphism $L\pi^* L_\cY \to L_\cX$ induces a morphism $H^0(L\varphi^*L\pi^*L_\cY)_\tor \to H^0(L\varphi^* L_\cX)_\tor$ of $k'\llbracket t \rrbracket$-modules, where by slight abuse of notation we also let $\varphi$ denote the corresponding map $\Spec(k'\llbracket t \rrbracket) \to \cX$. If $\het^{(1)}_{L_{\cX/\cY}}(\varphi) \neq \infty$, set
\begin{align*}
	\het_{\cX/\cY}(\varphi) = &\het^{(0)}_{L_{\cX/\cY}}(\varphi) - \het^{(1)}_{L_{\cX/\cY}}(\varphi)\\
	&- \dim_{k'}\coker\left(H^0(L\varphi^*L\pi^*L_\cY)_\tor \to H^0(L\varphi^* L_\cX)_\tor\right),
\end{align*}
and otherwise set $\het_{\cX/\cY}(\varphi) = \infty$. This induces a map
\[
	\het_{\cX/\cY}: |\sL(\cX)| \to \Z \cup \{\infty\}
\]
that we call the \emph{height function of $\cX$ over $\cY$}.
\end{definition}

\begin{remark}
\label{remarkHeightFunctionCompatibleWithOpenSubstack}
It is straightforward to check that if $\cU$ is an open substack of $\cX$, then the restriction of $\het_{\cX/\cY}$ to $|\sL(\cU)|$ coincides with $\het_{\cU/\cY}$.
\end{remark}

We begin by showing that $\het_{\cX/\cY}$ is invariant under finite field extensions, and hence does indeed define a function on $|\sL(\cX)|$.

\begin{lemma}\label{lemmaHeightFunctionsOfCotangentComplexFieldExt}
With notation as in \autoref{def:relhet}, $\het_{\cX/\cY}$ is a well-defined function on $|\sL(\cX)|$.
\end{lemma}
\begin{proof}
Since $\het^{(i)}_{L_{\cX/\cY}}$ is well-defined, it suffices to prove that the quantity $$\dim_{k'}\coker\left(H^0(L\varphi^*L\pi^*L_\cY)_\tor \to H^0(L\varphi^* L_\cX)_\tor\right)$$ is independent under taking finite extensions of $k'$ (which are necessarily \'etale since we are working in characteristic $0$.)

For this, let $R' = k'\llbracket t \rrbracket$ and $R'' := k''\llbracket t \rrbracket = R' \otimes_{k'} k''$. Let $K' = k'\llparenthesis t\rrparenthesis$ and $K'' = k''\llparenthesis t\rrparenthesis = K' \otimes_{k'} k''$. Now, for any $R'$-module $M$, we have an exact sequence
\[
0\to M_{\tor}\to M\to M\otimes_{R'}K'\to 0.
\]
Tensoring with $R''$ (which is \'etale over $R'$) shows that 
\[
M_{\tor}\otimes_{R'}R''=(M\otimes_{R'}R'')_{\tor}.
\]
Hence, given a map $M\to N$ of $R'$-modules and letting $Q$ be the cokernel of $M_{\tor}\to N_{\tor}$, we see
\[
Q\otimes_{R'}R''=\coker((M\otimes_{R'}R'')_{\tor}\to (N\otimes_{R'}R'')_{\tor}).
\]
Applying this to the case $M=H^0(L\varphi^*L\pi^*L_\cY)$ and $N=H^0(L\varphi^* L_\cX)$ proves the result.
\end{proof}

The following lemma allows us to understand the relative height function when the source is smooth.

\begin{lemma}
\label{lemmaHeightFunctionsOfCotangentComplexOfSmoothStack}
Let $\cZ$ be a smooth Artin stack over $k$, let $k'$ be a field extension of $k$, and let $\varphi: \Spec(k'\llbracket t \rrbracket) \to \cZ$ be a map. Then $H^0(L\varphi^* L_\cZ)$ is a torsion-free $k'\llbracket t \rrbracket$-module and $H^i(L\varphi^*L_\cZ) = 0$ for all $i \neq 0, 1$.
\end{lemma}

\begin{proof}
Let $\xi: Z \to \cZ$ be a smooth cover by a scheme. Possibly after extending $k'$, we may assume there exists some $\psi: \Spec(k'\llbracket t \rrbracket) \to Z$ such that $\xi \circ \psi \cong \varphi$. Applying $L\psi^*$ to the exact triangle
\[
	L\xi^* L_\cZ \to \Omega^1_Z \to \Omega^1_{Z/\cZ},
\]
we have an exact triangle
\[
	L\varphi^*L_\cZ \to \psi^* \Omega^1_Z \to \psi^* \Omega^1_{Z/\cZ}.
\]
The desired result then follows immediately from the fact that $\psi^* \Omega^1_Z$ is free.
\end{proof}

\begin{corollary}\label{cor:hetwithsmoothsource}
Let $\pi: \cX \to \cY$ be a morphism of locally finite type Artin stacks over $k$. If $\cX$ is smooth, then 
\[
	\het_{\cX/\cY} = \het^{(0)}_{L_{\cX/\cY}} - \het^{(1)}_{L_{\cX/\cY}}.
\]
\end{corollary}
\begin{proof}
This is immediate from \autoref{lemmaHeightFunctionsOfCotangentComplexOfSmoothStack} as $H^0(L\varphi^* L_\cX)_\tor$ vanishes.
\end{proof}

We end this section by showing that the relative height function is additive under compositions of morphisms.

\begin{proposition}\label{prop:additiveht-exact-triangle}
Let $\cX\xrightarrow{f}\cY\xrightarrow{g}\cZ$ be morphisms of stack. Assume there is an open substack $\cV\subset\cX$ such that $\cV\to\cY$ and $\cV\to\cZ$ are open immersions. If $\cV$ is smooth, then for all $\varphi\in\sL(\cX)\setminus\sL(\cX\setminus\cV)$, we have
\[
\het_{\cX/\cZ}(\varphi)=\het_{\cX/\cY}(\varphi)+\het_{\cY/\cZ}(f\varphi)
\]
\end{proposition}
\begin{proof}
We begin by reducing to the case where $\cX$ is smooth. When $\cX$ is not smooth, by \autoref{propositionResOfSingForStacks}, we may choose a representable resolution of singularities $\rho\colon\widetilde{\cX}\to\cX$ with $\rho|_\cV$ an isomorphism. Then $\varphi$ lifts to an arc $\widetilde{\varphi}$ of $\widetilde{X}$. We then have
\begin{align*}
\het_{\cX/\cY}(\varphi) & = \het_{\widetilde{\cX}/\cY}(\widetilde{\varphi}) - \het_{\widetilde{\cX}/\cX}(\widetilde{\varphi})\\
\het_{\cX/\cZ}(\varphi) & = \het_{\widetilde{\cX}/\cZ}(\widetilde{\varphi}) - \het_{\widetilde{\cX}/\cX}(\widetilde{\varphi})\\
\het_{\cY/\cZ}(f\varphi) & = \het_{\widetilde{\cX}/\cZ}(\widetilde{\varphi}) - \het_{\widetilde{\cX}/\cY}(\widetilde{\varphi})
\end{align*}
from which the result follows.

It remains to handle the case when $\cX$ is smooth. By \autoref{cor:hetwithsmoothsource}, we must show 
\[
	-\het_{\cY / \cZ} (\sL(f)(\varphi)) = -\het^{(0)}_{L_{\cX/\cZ}}(\varphi) + \het^{(1)}_{L_{\cX/\cZ}}(\varphi) + \het^{(0)}_{L_{\cX/\cY}}(\varphi)-\het^{(1)}_{L_{\cX/\cY}}(\varphi).
\]
Consider the exact triangle
\[
	Lf^*L_{\cY / \cZ} \to L_{\cX / \cZ} \to L_{\cX / \cY}.
\]
Pulling back along $\varphi$ we get the exact triangle
\[
	L\varphi^*Lf^*L_{\cY/\cZ} \to L\varphi^*L_{\cX/\cZ} \to L\varphi^* L_{\cX/\cY},
\]
which gives the exact sequence
\begin{align*}
	&H^{-1}(L\varphi^*L_{\cX/\cZ}) \to H^{-1}(L\varphi^* L_{\cX/\cY}) \\
	&\to H^0(L\varphi^*Lf^*L_{\cY/\cZ}) \to H^{0}(L\varphi^*L_{\cX/\cZ}) \to H^{0}(L\varphi^* L_{\cX/\cY})\\
	&\to H^1(L\varphi^*Lf^*L_{\cY/\cZ}) \to H^{1}(L\varphi^*L_{\cX/\cZ}) \to H^{1}(L\varphi^*L_{\cX/\cY}) \to 0.
\end{align*}
The restrictions of $Lf^*L_{\cY / \cZ}, L_{\cX / \cZ}$, and $L_{\cX / \cY}$ to $\cV$ are trivial and $\varphi$ is not in $\sL(\cX \setminus \cV)$, so these $k'\llbracket t \rrbracket$-modules are finite dimensional over $k'$. Thus
\begin{align*}
	0 = &\dim_{k'}\coker\left( H^{-1}(L\varphi^*L_{\cX/\cZ}) \to H^{-1}(L\varphi^* L_{\cX/\cY}) \right)\\
	&- \het^{(0)}_{L_{\cY/\cZ}}(\sL(f)(\varphi)) + \het^{(0)}_{L_{\cX/\cZ}}(\varphi) - \het^{(0)}_{L_{\cX/\cY}}(\varphi)\\
	&+\het^{(1)}_{L_{\cY/\cZ}}(\sL(f)(\varphi)) - \het^{(1)}_{L_{\cX/\cZ}}(\varphi)+\het^{(1)}_{L_{\cX/\cY}}(\varphi).
\end{align*}
We are therefore done if we can prove that
\begin{align*}
	\dim_{k'}&\coker\left( H^{-1}(L\varphi^*L_{\cX/\cZ}) \to H^{-1}(L\varphi^* L_{\cX/\cY}) \right) \\
	&= \dim_{k'}\coker\left(H^0(L\varphi^*Lf^*Lg^*L_\cZ)_\tor \to H^0(L\varphi^*Lf^* L_\cY)_\tor\right).
\end{align*}
Consider the commuting diagram
\begin{center}
\begin{tikzcd}
Lf^*Lg^*L_\cZ \arrow[r] \arrow[d] & L_\cX \arrow[equal,d] \arrow[r] &L_{\cX/\cZ} \arrow[d]\\
Lf^*L_\cY \arrow[r] & L_\cX \arrow[r] &L_{\cX/\cY}
\end{tikzcd}
\end{center}
whose rows are exact triangles. Noting that $H^{-1}(L\varphi^*L_\cX) = 0$ by \autoref{lemmaHeightFunctionsOfCotangentComplexOfSmoothStack}, pulling back along $\varphi$ and taking cohomology gives the commuting diagram
\begin{center}
\begin{tikzcd}
0 \arrow[r] \arrow[equal, d] & H^{-1}(L\varphi^*L_{\cX/\cZ}) \arrow[r] \arrow[d] & H^0 (L\varphi^*Lf^*Lg^*L_\cZ) \arrow[d] \arrow[r] & H^0 (L\varphi^* L_\cX) \arrow[equal, d]\\
0 \arrow[r] & H^{-1}(L\varphi^*L_{\cX/\cY}) \arrow[r] & H^0 (L\varphi^*Lf^*L_\cY) \arrow[r] & H^0 (L\varphi^* L_\cX)
\end{tikzcd}
\end{center}
whose rows are exact. Recall that because $\varphi$ is not in $\sL(\cX \setminus \cV)$, the $k'\llbracket t \rrbracket$-modules  $H^{-1}(L\varphi^*L_{\cX/\cZ})$ and $H^{-1}(L\varphi^* L_{\cX/\cY})$ are torsion. Also the $k'\llbracket t \rrbracket$-module $H^0 (L\varphi^* L_\cX)$ is torsion-free by \autoref{lemmaHeightFunctionsOfCotangentComplexOfSmoothStack}. Thus the middle horizontal arrows are just the inclusions of the respective torsion submodules. The commutativity of the diagram then implies that the induced map
\begin{align*}
	\coker&\left( H^{-1}(L\varphi^*L_{\cX/\cZ}) \to H^{-1}(L\varphi^* L_{\cX/\cY}) \right)\\
	&\xrightarrow{\simeq} \coker\left(H^0(L\varphi^*Lf^*Lg^*L_\cZ)_\tor \to H^0(L\varphi^*Lf^* L_\cY)_\tor\right),
\end{align*}
is an isomorphism.
\end{proof}

\section{A variant on \autoref{theoremMotivicChangeOfVariablesMeasurable}}

Our goal in this section is to prove the following motivic change of variables formula.

\begin{theorem}\label{theoremMCVFCylinder}
Let $\cX$ be a locally finite type Artin stack over $k$ with affine geometric stabilizers and separated diagonal, let $Y$ be an irreducible finite type scheme over $k$, let $\pi: \cX \to Y$ be a morphism, let $\cU$ be a smooth open substack of $\cX$ such that $\cU \hookrightarrow \cX \xrightarrow{\pi} Y$ is an open immersion, let $\cC \subset |\sL(\cX)|$ be a cylinder that is disjoint from $|\sL(\cX \setminus \cU)|$, and let $D \subset \sL(Y)$ be a cylinder such that $\pi$ induces a bijection $\overline{\cC}(k') \to D(k')$ for every field extension $k'$ of $k$.

\begin{enumerate}[(a)]

\item The restriction of $\het_{\cX/Y}$ to $\cC$ is integer valued and takes only finitely many values.

\item For all $n \in \Z$, the set $\het_{\cX / Y}^{-1}(n) \cap \cC$ is a $\dim Y$-convergent cylinder in $|\sL(\cX)|$.

\item We have an equality
\[
	\mu_Y(D) = \int_\cC \bL^{-\het_{\cX/Y}} \diff\mu_{\cX, \dim Y}.
\]

\end{enumerate}
\end{theorem}

We turn to the proof after a preliminary result.

\begin{lemma}
\label{lemmaBijectionAutomaticallyMakesCylinderSmall}
Let $\cX$ and $\cY$ be locally finite type Artin stacks over $k$, let $\pi: \cX \to \cY$ be a morphism, let $\cU$ be a smooth open substack of $\cX$ such that $\cU \hookrightarrow \cX \xrightarrow{\pi} \cY$ is an open immersion, let $\cC \subset |\sL(\cX)|$ be a cylinder that is disjoint from $|\sL(\cX \setminus \cU)|$, let $\cD \subset |\sL(\cY)|$ be a cylinder, assume that $\cY$ is equidimensional, and assume that $\pi$ induces a bijection $\overline{\cC}(k') \to \overline{\cD}(k')$ for every field extension $k'$ of $k$. Then $\cC$ is small if and only if $\cD$ is small.
\end{lemma}

\begin{proof}
Because $\theta_0(\cD)$ is the image of $\theta_0(\cC)$ under the continuous map $|\cX| \to |\cY|$, it is clear that if $\cC$ is small then $\cD$ is small.

For the other direction, assume that $\cD$ is small. For every $\varphi \in \cC$ choose a quasi-compact open substack $\cX_\varphi$ of $\cX$ such that $\theta_0(\varphi) \in |\cX_\varphi|$, and set $\cC_\varphi = \cC \cap \theta_0^{-1}(|\cX_\varphi|)$ and $\cD_\varphi = \sL(\pi)(\cC_\varphi)$. Then $\cC = \bigcup_{\varphi \in \cC} \cC_\varphi$ and each $\cC_\varphi$ is a small cylinder that is disjoint from $|\sL(\cX \setminus \cU)|$. So, $\cD = \bigcup_{\varphi \in \cC} \cD_\varphi$ as $\overline{\cC}(k') \to \overline{\cD}(k')$ is surjective, and each $\cD_\varphi$ is a small cylinder by \autoref{propositionImageOfCylinderIsCylinder}. By \autoref{propositionSmallCylinderFiniteSubcover}, there exist $\varphi_1, \dots, \varphi_r \in \cC$ such that $\cD = \cD_{\varphi_1} \cup \dots \cup \cD_{\varphi_r}$. Because $\overline{\cC}(k') \to \overline{\cD}(k')$ is injective for every field extension $k'$ of $k$, it is straightforward to check that $\cC = \cC_{\varphi_1} \cup \dots \cup \cC_{\varphi_r}$, which is small.
\end{proof}

\begin{proof}[{Proof of \autoref{theoremMCVFCylinder}}]
Because the theorem is trivially true when $\cU$ is empty, we will assume that $\cU$ is nonempty. In particular $\cU$ is integral and $\dim\cU = \dim Y$.

Because $Y$ is finite type, $D$ is a small, so $\cC$ is small by \autoref{lemmaBijectionAutomaticallyMakesCylinderSmall}. Because $Y$ is finite type and $\cU$ is isomorphic to an open subscheme of $Y$, we have $\cU$ is quasi-compact. Replacing $\cX$ with a quasi-compact open substack that contains $\theta_0(\cC)$ and $\cU$ and noting \autoref{remarkHeightFunctionCompatibleWithOpenSubstack}, we may assume that $\cX$ is finite type over $k$.

Let $\cX'$ be the scheme theoretic image of $\cU \hookrightarrow \cX$. Then $\cX'$ is integral and finite type over $k$, so by \autoref{propositionResOfSingForStacks} there exists an irreducible smooth finite type Artin stack $\cZ$ over $k$ and a morphism $\rho': \cZ \to \cX'$ that is proper and representable by schemes and such that $\rho'^{-1}(\cU) \to \cU$ is an isomorphism. Let $\rho: \cZ \to \cX$ be the composition $\cZ \xrightarrow{\rho'} \cX' \hookrightarrow \cX$, and set $\cE = \sL(\rho)^{-1}(\cC)$. Noting \autoref{lemmaNotArcOfBoundaryImpliesTruncationInClosure}, we have $\cX, \cZ, \rho, \cU, \cC, \cE$ satisfy the hypotheses of \autoref{theoremRepresentableChangeOfVariables}. By \autoref{theoremRepresentableChangeOfVariables}(\ref{theoremPartCylinderBijection}), the hypotheses of \cite[Theorem 1.3]{SatrianoUsatine2} hold, hence combining with \cite[Theorem 2.3]{SatrianoUsatine3}, we see $-\het^{(0)}_{L_{\cZ/Y}} + \het^{(1)}_{L_{\cZ/Y}}$ is integer valued and constructible on $\cE$ and
\[
	\mu_Y(D) = \int_{\cE} \bL^{-\het^{(0)}_{L_{\cZ/Y}} + \het^{(1)}_{L_{\cZ/Y}}}\diff\mu_\cZ.
\]
For each $n, m \in \Z$ set
\[
	\cE^{(n, m)} = \cE \cap (-\het^{(0)}_{L_{\cZ/Y}} + \het^{(1)}_{L_{\cZ/Y}})^{-1}(n) \cap (\het^{(0)}_{L_{\cZ/\cX}})^{-1}(m)
\]
and
\[
	\cC^{(n,m)} = \sL(\rho)(\cE^{(n,m)}).
\]
We have each $\cE^{(n,m)}$ is a cylinder, so each $\cC^{(n,m)}$ is a cylinder in $|\sL(\cX')|$ by \autoref{propositionImageOfCylinderIsCylinder}. 

We will show that each $\cC^{(n,m)}$ is a cylinder in $|\sL(\cX)|$. There exists some $\ell \in \Z_{\geq 0}$ and constructible subset $\cC_\ell \subset |\sL_\ell(\cX')|$ such that 
\[
	\cC^{(n,m)} = \theta_{\ell, \cX'}^{-1}(\cC_\ell) = \theta_{\ell, \cX}^{-1}(\cC_\ell) \cap |\sL(\cX')|.
\]
Intersecting with $\cC$ on both sides and noting $\cC \subset  |\sL(\cX')|$ and $\cC^{(n,m)} \subset \cC$,
\[
	\cC^{(n,m)} = \theta_{\ell, \cX}^{-1}(\cC_\ell) \cap \cC,
\]
which is a cylinder in $|\sL(\cX)|$.

Now for all $n,m$, we have $\cX, \cZ, \rho, \cU, \cC^{(n,m)}, \cE^{(n,m)}$ satisfy the hypotheses of \autoref{theoremRepresentableChangeOfVariables}, so by \autoref{theoremRepresentableChangeOfVariables}(\ref{theoremPartIntegralOfMCVF}) each $\cC^{(n,m)}$ is $\dim Y$-convergent and
\[
	\mu_{\cX, \dim Y}(\cC^{(n,m)}) = \int_{\cE^{(n,m)}} \bL^{-\het^{(0)}_{L_{\cZ/\cX}}}\diff\mu_\cZ = \bL^{-m}\mu_\cZ(\cE^{(n,m)}).
\]
Therefore
\[
	\mu_Y(D) = \sum_{n,m \in \Z} \bL^n \mu_\cZ(\cE^{(n,m)}) = \sum_{n,m \in \Z} \bL^{n+m} \mu_{\cX, \dim Y}(\cC^{(n,m)}).
\]
Thus we are done if we can show that for all $n, m \in \Z$, the function $-\het_{\cX / Y}$ is equal to $n + m$ on all of $\cC^{(n,m)}$. This follows since
\[
	-\het_{\cX / Y} (\sL(\rho)(\psi)) = -\het^{(0)}_{L_{\cZ/Y}}(\psi) + \het^{(1)}_{L_{\cZ/Y}}(\psi) + \het^{(0)}_{L_{\cZ/\cX}}(\psi),
\]
by \autoref{cor:hetwithsmoothsource} and \autoref{prop:additiveht-exact-triangle} applied to $\cZ\to\cX\to Y$.
\end{proof}

\section{Motivic change of variables formula:~proof of \autoref{theoremMotivicChangeOfVariablesMeasurable}}

Throughout this subsection, let $\cX, Y, \pi, U, \cU, \cC$ be as in the hypotheses of \autoref{theoremMotivicChangeOfVariablesMeasurable}, and set $d = \dim Y$. 

We first prove our desired result for measurable sets contained in a cylinder.

\begin{lemma}
\label{lemmaMCVFMeasurableInsideCylinder}
Let $D \subset \sL(Y)$ be a measurable set, let $E$ be a cylinder contained in $\sL(Y) \setminus \sL(Y \setminus U)$, and assume $D \subset E$. Then the restriction of $\het_{\cX/Y}$ to $\cC \cap \sL(\pi)^{-1}(D)$ is integer valued and takes only finitely many values, $\bL^{-\het_{\cX/Y}}$ is $d$-integrable on $\cC \cap \sL(\pi)^{-1}(D)$, and
\[
	\mu_Y(D) = \int_{\cC \cap \sL(\pi)^{-1}(D)} \bL^{-\het_{\cX/Y}} \diff\mu_{\cX, d}.
\]
\end{lemma}

\begin{proof}
For all $n \in \Z$, set $\cC_n = \cC \cap \sL(\pi)^{-1}(D) \cap \het_{\cX/Y}^{-1}(n)$. For all $\varepsilon \in \R_{> 0}$, choose a cylindrical $\varepsilon$-approximation $(D_\varepsilon^{(0)}, \{D_\varepsilon^{(i)}\}_{i \in I_\varepsilon})$ of $D$. Since $D \subset E$ and noting e.g., \autoref{lemmaNormInclusionCylinders}, we may replace each $D_\varepsilon^{(0)}$ and $D_\varepsilon^{(i)}$ with their intersections with $E$ and therefore assume that all $D_\varepsilon^{(0)}$ and $D_\varepsilon^{(i)}$ are contained in $E$. For all $\varepsilon \in \R_{> 0}$ and $n \in \Z$, set
\[
	\cC_{\varepsilon, n}^{(0)} = \cC \cap \sL(\pi)^{-1}(D_\varepsilon^{(0)}) \cap \het_{\cX/Y}^{-1}(n),
\]
and for all $i \in I_\varepsilon$, set
\[
	\cC_{\varepsilon, n}^{(i)} = \cC  \cap \sL(\pi)^{-1}(D_\varepsilon^{(i)}) \cap \het_{\cX/Y}^{-1}(n).
\]
We will show that for all $\varepsilon, n$, we have $(\cC_{\varepsilon, n}^{(0)}, \{\cC_{\varepsilon, n}^{(i)}\}_{i \in I_\varepsilon})$ is a $d$-convergent cylindrical $\varepsilon\exp(n)$-approximation of $\cC_n$. It is straightforward to check that
\[
	(\cC_n \setminus \cC^{(0)}_{\varepsilon, n}) \cup (\cC^{(0)}_{\varepsilon, n} \setminus \cC_n) \subset \bigcup_{i \in I_\varepsilon} \cC^{(i)}_{\varepsilon, n}.
\]
Applying \autoref{theoremMCVFCylinder} to $\cC \cap \sL(\pi)^{-1}(E) \to E$ gives that the restriction of $\het_{\cX/Y}$ to $\cC \cap \sL(\pi)^{-1}(E)$ is integer valued and takes only finitely many values. Because $D \subset E$, we also have that the restriction of $\het_{\cX/Y}$ to $\cC \cap \sL(\pi)^{-1}(D)$ is integer valued and takes only finitely many values. 

Let $n_0 \in \Z_{\geq 0}$ be the maximum of the absolute value of $\het_{\cX/Y}$ on $\cC \cap \sL(\pi)^{-1}(E)$. In particular, for all $\varepsilon, i$, we have $\cC_n$, $\cC^{(0)}_{\varepsilon, n}$ and $\cC^{(i)}_{\varepsilon, n}$ are empty whenever $|n| > n_0$. For each $\varepsilon, i$, applying \autoref{theoremMCVFCylinder} to $\cC \cap \sL(\pi)^{-1}(D^{(i)}_\varepsilon) \to D^{(i)}_\varepsilon$ gives that 
\begin{itemize}

\item each $\cC^{(i)}_{\varepsilon, n}$ is a $d$-convergent cylinder in $|\sL(\cX)|$, and

\item we have the equality
\[
	\mu_Y(D^{(i)}_\varepsilon) = \sum_{n = -n_0}^{n_0} \bL^{-n} \mu_{\cX, d}(\cC^{(i)}_{\varepsilon, n}).
\]

\end{itemize}
Then by \autoref{propositionNormsAndSumsOfLimits} and \autoref{remarkMeasureWrittenAsEffectiveLimit},
\[
	\Vert \mu_{\cX, d}(\cC^{(i)}_{\varepsilon, n}) \Vert = \exp(n)\Vert \bL^{-n} \mu_{\cX, d}(\cC^{(i)}_{\varepsilon, n})\Vert \leq \exp(n) \Vert \mu_Y(D^{(i)}_\varepsilon) \Vert < \varepsilon\exp(n).
\]
Similarly, for all $\varepsilon$
\begin{itemize}

\item each $\cC^{(0)}_{\varepsilon, n}$ is a $d$-convergent cylinder in $|\sL(\cX)|$, and

\item we have the equality
\[
	\mu_Y(D^{(0)}_\varepsilon) = \sum_{n = -n_0}^{n_0} \bL^{-n} \mu_{\cX, d}(\cC^{(0)}_{\varepsilon, n}).
\]

\end{itemize}
In particular, for all $\varepsilon, n$ we have $(\cC_{\varepsilon, n}^{(0)}, \{\cC_{\varepsilon, n}^{(i)}\}_{i \in I_\varepsilon})$ is a $d$-convergent cylindrical $\varepsilon\exp(n)$-approximation of $\cC_n$. Thus each $\cC_n$ is $d$-measurable and
\[
	\mu_{\cX,d}(\cC_n) = \lim_{\varepsilon \to 0} \mu_{\cX,d}(\cC^{(0)}_{\varepsilon, n}).
\]
Since $\cC_n$ is empty for all but finitely many $n$, this also implies that $\bL^{-\het_{\cX/Y}}$ is $d$-integrable on $\cC \cap \sL(\pi)^{-1}(D)$. Finally
\begin{align*}
	\int_{\cC \cap \sL(\pi)^{-1}(D)} \bL^{-\het_{\cX/Y}} \diff\mu_{\cX, d} &= \sum_{n = -n_0}^{n_0} \bL^{-n} \mu_{\cX, d}(\cC_n)\\
	&= \sum_{n = -n_0}^{n_0}\bL^{-n} \lim_{\varepsilon \to 0} \mu_{\cX,d}(\cC^{(0)}_{\varepsilon, n})\\
	&= \lim_{\varepsilon \to 0} \sum_{n = -n_0}^{n_0} \bL^{-n} \mu_{\cX, d}(\cC^{(0)}_{\varepsilon, n})\\
	&= \lim_{\varepsilon \to 0} \mu_Y(D^{(0)}_\varepsilon) = \mu_Y(D).
\end{align*}
\end{proof}

We next loosen the restriction that $D$ be contained in a cylinder.

\begin{lemma}
\label{lemmaMCVFSingleMeasurableSet}
Let $D \subset \sL(Y)$ be a measurable set. Then $\bL^{-\het_{\cX/Y}}$ is $d$-integrable on $(\cC \setminus |\sL(\cX \setminus \cU)|) \cap \sL(\pi)^{-1}(D)$ and
\[
	\mu_Y(D) = \int_{(\cC \setminus |\sL(\cX \setminus \cU)|) \cap \sL(\pi)^{-1}(D)} \bL^{-\het_{\cX/Y}} \diff\mu_{\cX, d}.
\]
\end{lemma}

\begin{proof}
Throughout this proof, choose some closed subscheme structure for $Y \setminus U \hookrightarrow Y$, set
\[
	E_0 = \sL(Y) \setminus \theta_0^{-1}(\sL_0(Y \setminus U)),
\] 
and for all $\ell \in \Z_{\geq 1}$ set 
\[
	E_\ell = \theta_{\ell-1}^{-1}(\sL_{\ell-1}(Y \setminus U)) \setminus \theta_{\ell}^{-1}(\sL_{\ell}(Y\setminus U)) \subset \sL(Y).
\]
Therefore
\[
	\sL(Y) \setminus \sL(Y \setminus U) = \bigsqcup_{\ell \in \Z_{\geq 0}} E_\ell.
\]
Because $\mu_Y(\sL(Y \setminus U)) = 0$, we may replace $D$ with $D \setminus \sL(Y \setminus U)$ and therefore assume that $D$ is disjoint from $\sL(Y \setminus U)$. We then have
\[
	(\cC \setminus |\sL(\cX \setminus \cU)|) \cap \sL(\pi)^{-1}(D) = \cC \cap \sL(\pi)^{-1}(D),
\]
and
\[
	D = \bigsqcup_{\ell \in \Z_{\geq 0}} D \cap E_\ell.
\]
Because each $E_\ell$ is a cylinder and therefore measurable, each $D \cap E_\ell$ is measurable. Then, e.g., by \autoref{propositionConditionForUnionOfMeasurablesIsMeasurable},
\[
	\mu_Y(D) = \sum_{\ell \in \Z_{\geq 0}} \mu_Y(D \cap E_\ell).
\]
Also by \autoref{lemmaMCVFMeasurableInsideCylinder} applied to $D \cap E_\ell \subset E_\ell$,
\begin{itemize}

\item the restriction of $\het_{\cX/Y}$ to $\cC \cap \sL(\pi)^{-1}(D \cap E_\ell)$ is integer valued and takes finitely many values,

\item $\bL^{-\het_{\cX/Y}}$ is $d$-integrable on $\cC \cap \sL(\pi)^{-1}(D \cap E_\ell)$, and

\item we have the equality
\[
	\mu_Y(D \cap E_\ell) = \int_{\cC \cap \sL(\pi)^{-1}(D \cap E_\ell)} \bL^{-\het_{\cX/Y}} \diff\mu_{\cX,d}.
\]

\end{itemize}
For each $\ell \in \Z_{\geq 0}$, let $n_\ell$ be the maximum of the absolute value of $\het_{\cX/Y}$ on $\cC \cap \sL(\pi)^{-1}(D \cap E_\ell)$. Then combining the above,
\begin{align*}
	\mu_Y(D) &= \sum_{\ell \in \Z_{\geq 0}} \mu_Y(D \cap E_\ell)\\
	&= \sum_{\ell \in \Z_{\geq 0}} \int_{\cC \cap \sL(\pi)^{-1}(D \cap E_\ell)} \bL^{-\het_{\cX/Y}} \diff\mu_{\cX,d}\\
	&= \sum_{\ell \in \Z_{\geq 0}} \sum_{n = -n_\ell}^{n_\ell} \bL^{-n} \mu_{\cX, d}(\cC \cap \sL(\pi)^{-1}(D \cap E_\ell) \cap \het_{\cX/Y}^{-1}(-n)).
\end{align*}
The convergence of the final series implies that for any $n \in \Z$,
\[
	\lim_{\ell \to \infty} \mu_{\cX, d}(\cC \cap \sL(\pi)^{-1}(D \cap E_\ell) \cap \het_{\cX/Y}^{-1}(-n)) = 0.
\]
Thus by \autoref{propositionConditionForUnionOfMeasurablesIsMeasurable} for any $n \in \Z$,
\[
	\cC \cap \sL(\pi)^{-1}(D) \cap \het_{\cX/Y}^{-1}(-n)= \bigsqcup_{\ell \in \Z_{\geq 0}}(\cC \cap \sL(\pi)^{-1}(D \cap E_\ell) \cap \het_{\cX/Y}^{-1}(-n))
\]
is $d$-measurable and
\begin{align*}
	\mu_{\cX, d}(\cC \cap \sL(\pi)^{-1}(D) &\cap \het_{\cX/Y}^{-1}(-n)) \\
	&= \sum_{\ell \in \Z_{\geq 0}} \mu_{\cX, d}(\cC \cap \sL(\pi)^{-1}(D \cap E_\ell) \cap \het_{\cX/Y}^{-1}(-n)).
\end{align*}
Therefore, if for each $n\in\Z_{\geq0}$, we set $\mathcal{S}_n=\{\ell\mid |n|\leq n_\ell\}$, we find
\begin{align*}
	\int_{\cC \cap \sL(\pi)^{-1}(D)} &\bL^{-\het_{\cX/Y}}\diff\mu_{\cX, d} \\
	&= \sum_{n \in \Z_{\geq 0}} \sum_{\ell \in \mathcal{S}_n} \bL^{-n} \mu_{\cX, d}(\cC \cap \sL(\pi)^{-1}(D \cap E_\ell) \cap \het_{\cX/Y}^{-1}(-n))\\
	&= \mu_Y(D),
\end{align*}
and we are done.
\end{proof}

%For the remainder of this subsection, let $f$ be as in the hypotheses of \autoref{theoremMotivicChangeOfVariablesMeasurable}.

Lastly, we prove the following more general form of \autoref{theoremMotivicChangeOfVariablesMeasurable}.

\begin{theorem}\label{theoremMotivicChangeOfVariablesMeasurable-realversion}
Keep Notation \ref{not:mainnot}, and let $\cC \subset |\sL(\cX)|$ be a cylinder such that the map 
\[
%\overline{
\cC \setminus |\sL(\cX \setminus \cU)|\ \longrightarrow\ \sL(Y) \setminus \sL(Y \setminus U)
\]
induced by $\pi$ is a bijection on isomorphism classes of $k'$-points for all field extensions $k'$ of $k$.
Let $f: \sL(Y) \to \Z \cup \{\infty\}$ be a measurable function.

\begin{enumerate}[(a)]

\item\label{theoremMotivicChangeOfVariablesMeasurable-realversion::a} We have $\bL^f$ is integrable on $\sL(Y)$ if and only if $\bL^{f \circ \sL(\pi)-\het_{\cX/Y}}$ is $\dim Y$-integrable on $\cC \setminus |\sL(\cX \setminus \cU)|$.

\item If $\bL^f$ is integrable on $\sL(Y)$, then
\[
	 \int_{\sL(Y)} \bL^{f} \diff\mu_Y = \int_{\cC \setminus |\sL(\cX \setminus \cU)|} \bL^{f \circ \sL(\pi)-\het_{\cX/Y}} \diff\mu_{\cX, \dim Y}.
\]

\end{enumerate}
\end{theorem}

\begin{proof}%[Proof of \autoref{theoremMotivicChangeOfVariablesMeasurable}]
Set $\cC' = \cC \setminus |\sL(\cX \setminus \cU)|$, and for all $\ell \in \Z$, set $D_\ell = f^{-1}(\ell)$. Then by \autoref{lemmaMCVFSingleMeasurableSet}, $\bL^{-\het_{\cX/Y}}$ is $d$-integrable on $\cC' \cap \sL(\pi)^{-1}(D_\ell)$ and
\begin{align*}
	\mu_Y(D_\ell) &= \int_{\cC' \cap \sL(\pi)^{-1}(D_\ell)} \bL^{-\het_{\cX/Y}} \diff\mu_{\cX, d}\\
	&= \sum_{n \in \Z} \bL^{-n} \mu_{\cX, d}(\cC' \cap \sL(\pi)^{-1}(D_\ell) \cap \het_{\cX/Y}^{-1}(-n) )\\
	&= \sum_{n \in \Z} \bL^{-n} \mu_{\cX, d}(\cC' \cap (f \circ \sL(\pi))^{-1}(\ell) \cap \het_{\cX/Y}^{-1}(-n) ).
\end{align*}
Therefore $\bL^f$ is integrable if and only if
\begin{align*}
	\sum_{\ell \in \Z} &\sum_{n \in \Z} \bL^{\ell-n} \mu_{\cX, d}(\cC' \cap (f \circ \sL(\pi))^{-1}(\ell) \cap \het_{\cX/Y}^{-1}(-n) )\\
	&= \sum_{m \in \Z} \bL^{m} \sum_{\ell - n = m} \mu_{\cX, d}(\cC' \cap (f \circ \sL(\pi))^{-1}(\ell) \cap \het_{\cX/Y}^{-1}(-n) )
\end{align*}
converges, which by \autoref{propositionConditionForUnionOfMeasurablesIsMeasurable} is equivalent to $\bL^{f \circ \sL(\pi)-\het_{\cX/Y}}$ being $d$-integrable on $\cC'$, and in that case
\[
	 \int_{\sL(Y)} \bL^{f} \diff\mu_Y = \int_{\cC'} \bL^{f \circ \sL(\pi)-\het_{\cX/Y}} \diff\mu_{\cX, d}.
\]
This completes our proof.
\end{proof}

\bibliographystyle{alpha}
\bibliography{SingularMCVF.bib}

\end{document}